\documentclass[12pt]{amsart}

\usepackage{amssymb}

\usepackage{enumerate}
\usepackage{amsmath,amsthm,,amssymb}
\usepackage{hyperref}
\usepackage[USenglish]{babel}

\makeatletter
\@namedef{subjclassname@2010}{%
  \textup{2010} Mathematics Subject Classification}
\makeatother

\usepackage[T1]{fontenc}
\newtheorem{theorem}{Theorem}[section]
\newtheorem{corollary}[theorem]{Corollary}
\newtheorem{lemma}[theorem]{Lemma}
\newtheorem{proposition}[theorem]{Proposition}

\theoremstyle{definition}
\newtheorem{definition}[theorem]{Definition}
\newtheorem{remark}[theorem]{Remark}

\numberwithin{equation}{section}

\DeclareMathOperator{\tr}{tr}

\begin{document}

%%%%% To ease editing, for IMPAN journals add:

\baselineskip=17pt

%%%%%%%%%%%%%%%%

\title
{Interpolation of quasi noncommutative $L_p$-spaces}

\author {Juan Gu}
\address{Department of Mathematics, Harbin Institute of Technology, Harbin 150006, China, and Department of Mathematics, Heilongjiang University of Science and Technology, Harbin, 150022, China}
\email{gujuan-1977@163.com}
\author {Zhi Yin}
\address{Institute of Advanced Study in Mathematics, Harbin Institute of Technology, Harbin 150006, China}
\email{hustyinzhi@163.com}
\author {Haonan Zhang}
\address{Laboratoire de Math\'ematiques, Universit\'e de Franche-Comt\'e, 25030 Besan\c con, France and Institute of Mathematics, Polish Academy of Sciences, ul. \'Sniadeckich 8, 00-656 Warszawa, Poland}
\email{haonan.zhang@edu.univ-fcomte.fr}

\date{\today}

\begin{abstract}
	Let $\mathcal{M}$ be a ($\sigma$-finite) von Neumann algebra associated with a normal faithful state $\phi.$ We prove a complex interpolation result for a couple of two (quasi) Haagerup noncommutative $L_p$-spaces $L_{p_0} (\mathcal{M}, \phi)$ and $L_{p_1} (\mathcal{M}, \phi), 0< p_0 < p_1\leq \infty,$ which has further applications to the sandwiched $p$-R\'{e}nyi divergence.
\end{abstract}

\subjclass[2010]{Primary 46L51; Secondary 46M35, 81P45.}

\keywords{Noncommutative $L_p$-spaces, complex interpolation, quantum relative entropies, data processing inequality.}

\maketitle

\section{Introduction}

Let $\mathcal{M}$ be a semifinite von Neumann algebra equipped with a normal semifinite faithful (n.s.f.) trace $\tau$, and $\Lambda_p(\mathcal{M}, \tau)$ be the tracial noncommutative $L_p$-space associated with $(\mathcal{M}, \tau)$. It is well-known \cite{PX2003} that the complex interpolation space of a couple of two Banach spaces $( \Lambda_{p_0} (\mathcal{M}, \tau), \Lambda_{p_1} (\mathcal{M}, \tau))$ is isometric to $\Lambda_{p_\theta}(\mathcal{M}, \tau),$ i.e.,
\begin{equation}\label{eq:complexinterpolation1}
[\Lambda_{p_0} (\mathcal{M}, \tau), \Lambda_{p_1}(\mathcal{M}, \tau)]_\theta = \Lambda_{p_\theta}(\mathcal{M}, \tau),
\end{equation}
where $1/p_\theta = (1-\theta)/ p_0 + \theta/p_1, 1\leq p_0 < p_1 \leq \infty.$ Through this line of research, the complex interpolation for weighted noncommutative $L_p$-spaces \cite{RX2011}, noncommutative Hardy spaces \cite{P1992, BO2017} and noncommutative Orlicz spaces \cite{BC2012, BCO2016} have been studied. Because noncommutative $L_p$-spaces naturally appear in the study of noncommutative analysis, the interpolation of noncommutative $L_p$-spaces, together with the Riesz-Thorin theorem and Marcinkiewicz theorem, become fundamental and powerful tools in proving noncommutative martingale inequalities \cite{JX2003, PX1997}, maximal ergodic inequalities \cite{JX2007}, and the duality properties of noncommutative function spaces \cite{M2007}.

In the early 1990s, Xu \cite{X1990} showed that the above interpolation result  \eqref{eq:complexinterpolation1} holds for the couple of quasi-Banach spaces $( \Lambda_{p_0}(\mathcal{M}, \tau), \Lambda_{p_1}(\mathcal{M}, \tau)),0< p_0 < p_1 \leq \infty.$ However, it didn't draw too much attention to the community, contrary to the Banach space case $(1\le p_0 < p_1\le \infty)$. In this paper, we will extend Xu's results to a more general setting. Namely, let $\mathcal{M}$ be a $\sigma$-finite von Neumann algebra equipped with a normal faithful (n.f.) state $\phi,$ we would like to pursue a complex interpolation result for a couple of (quasi) noncommutative $L_p$-spaces associated with $(\mathcal{M}, \phi).$

There are two different but equivalent approaches of constructing noncommutative $L_p$-spaces associated with $(\mathcal{M}, \phi).$ One was introduced by Haagerup \cite{H1979}. We denote by $L_p(\mathcal{M}, \phi)$ the Haagerup noncommutative $L_p$-space (see Section \ref{sec:preliminary} for definition). However, it is well-known that for two different indices $p_0$ and $p_1,$ the natural intersection $L_{p_0}(\mathcal{M}, \phi) \cap L_{p_1}(\mathcal{M}, \phi)$ is trivial \cite{T1981}, thus the couple $(L_{p_0}(\mathcal{M}, \phi), L_{p_1}(\mathcal{M}, \phi))$ is not compatible for interpolation. To overcome this difficulty, Kosaki \cite{K1984} introduced another kind of noncommutative $L_p$-space, denoted by $C_p(\mathcal{M}, \phi),1\leq p \leq \infty$. It was defined as the complex interpolation of the couple $(h_\phi^{1/2} \mathcal{M} h_\phi^{1/2}, L_1(\mathcal{M}, \phi))$ (see Section \ref{sec:preliminary} for the definition of $h_\phi$). More precisely, for $1\leq p \le\infty$ we have
$$C_{p}(\mathcal{M}, \phi): = \left[h_\phi^{1/2} \mathcal{M} h_\phi^{1/2}, L_1(\mathcal{M}, \phi) \right]_{\frac{1}{p}},$$
where $C_\infty (\mathcal{M}, \phi) = h_\phi^{1/2} \mathcal{M} h_\phi^{1/2}.$
By the reiteration theorem \cite[Theorem 4.6.1]{BL1976}, one has \cite{K1984}
$$[C_{p_0} (\mathcal{M}, \phi), C_{p_1}(\mathcal{M}, \phi)]_\theta = C_{p_\theta}(\mathcal{M}, \phi),$$
where $\frac{1}{p_{\theta}}=\frac{1-\theta}{p_0}+\frac{\theta}{p_1}$ and $1\leq p_0 < p_1 \leq \infty.$ The crucial point of Kosaki's interpolation is that $C_\infty (\mathcal{M}, \phi)$ naturally embeds into $L_1(\mathcal{M}, \phi),$ which makes $C_\infty (\mathcal{M}, \phi)$ and $L_1(\mathcal{M}, \phi)$ a compatible couple of interpolation. Hence the following Kosaki's embedding $i_p (x) = h_\phi^{\frac{1}{2p}} x h_\phi^{\frac{1}{2p}}, x\in \mathcal{M}$ embeds $\mathcal{M}$ into $L_p(\mathcal{M}, \phi)$ (see Proposition \ref{prop:embed1}). In this way any pair of Haagerup's noncommutative $L_p$-spaces is compatible. 

In this paper, We show that for $0< p_0 < p_1 \leq \infty$,
\begin{equation}
[L_{p_0} (\mathcal{M}, \phi), L_{p_1} (\mathcal{M}, \phi)]_\theta = L_{p_\theta} (\mathcal{M}, \phi),
\end{equation}
where $1/p_\theta = (1-\theta)/p_0 + \theta/p_1$ (see Theorem \ref{thm:interpolation2}). The key idea of our proof is to use Xu's reversed H\"{o}lder's inequality (see Lemma \ref{lem:decomposition} and Corollary \ref{cor:cor of xu's result}) for tracial noncommutative $L_p$-spaces \cite{X1990} and a Hadamard three lines theorem for (quasi) Haagerup noncommutative $L_p$-spaces (see Proposition \ref{prop:three-line-Haagerup}).

Recently, people in the quantum information community introduced a noncommutative $L_p$-space, denoted by $L_{p, \sigma} (H)$ \cite{Beigi2013, OZ1999, KT2013}. Recall that for a fixed positive operator $\sigma \in B(H),$ the space $L_{p, \sigma} (H)$ is defined as the completion of $(B(H), \|\cdot\|_{p, \sigma}),$ where the (quasi) norm $\|\cdot\|_{p, \sigma}$ is defined as
$$\|x\|_{p, \sigma}:=  \left\| \sigma^{\frac{1}{2p}} x \sigma^{\frac{1}{2p}} \right\|_p, ~x \in B(H)$$
for $0 < p < \infty.$ Here $\|\cdot\|_p$ is the Schatten $p$-norm on $B(H).$ In fact, due to the Tomita-Takesaki theory \cite{Takesaki}, $L_{p, \sigma}(H)$ is nothing but a special case of Haagerup noncommutative $L_p$-space. And it was used by Beigi \cite{Beigi2013} to study the data processing inequality (DPI) of the sandwiched $p$-R\'{e}nyi divergence $\tilde{D}_p$ introduced by M\"{u}ller Lennert et al. \cite{MDSFT2013}. This motivates us to define the following generalization of $p$-R\'{e}nyi divergence for $p \in (0, 1) \cup (1, \infty)$:
\begin{equation}
\tilde{D}_p (\psi \| \phi) :  =  \frac{1}{p-1} \log \left\| h_\phi^{\frac{1-p}{2p}} h_\psi h_\phi^{\frac{1-p}{2p}} \right\|^p_{L_p(\mathcal{M}, \phi)},
\end{equation}
where $\phi, \psi$ are two normal faithful states on $\mathcal{M}.$ This definition fits well with the one by Jen\v{c}ov\'a \cite{Jencova2018} (for $1< p < \infty$) and the one by Berta et al. \cite{BST2016} (for $1/2 \leq p< 1$), where they firstly generalized the sandwiched $p$-R\'{e}nyi divergence $\tilde{D}_p$ to the general von Neumann algebra setting. As an application of our interpolation result, we can easily cover Beigi's results by adapting his argument to our setting. Moreover, we also derive a simple sufficient condition for the equality in DPI for all $p \in (0, 1) \cup (1, \infty)$ (see Theorem \ref{thm:sufficient}). We end this introduction by the following remark: We have realized that the complex interpolation for noncommutative $L_p$-spaces has been used in many works about the sandwiched $p$-R\'{e}nyi divergence, see for instance \cite{Beigi2013, BST2016, JRSWW2018}. However, they all concern with the Banach noncommutative $L_p$-spaces $(1\leq p < \infty).$ Our interpolation result for the quasi ones will serve as a great complement to this topic.

Our paper is organized as follows. In Section 2 we summarize necessary preliminaries on tracial noncommutative $L_p$-spaces, Haagerup noncommutative $L_p$-spaces, Haagerup's reduction theorem and complex interpolation of tracial noncommutative $L_p$-spaces. Section 3 is devoted to the Hadamard three lines theorem for Haggerup noncommutative $L_p$-spaces, which is the key ingredient for proving our main interpolation result. We also present a direct application of our Hadamard three lines theorem to matrix inequalities. In Section 4 we show the main result about the complex interpolation. Section 5 presents some applications to the sandwiched $p$-R\'{e}nyi divergence.

%%%%%%%%%%%%%%%%%%%%%%%%

\section{Preliminaries}
\label{sec:preliminary}

\subsection{Tracial noncommutative $L_p$-spaces}

In this subsection we concentrate ourselves to noncommutative $L_p$-spaces associated with semifinite von Neumann algebras, which were firstly laid out in the early 1950's by Segal \cite{Segal1953} and Dixmier \cite{Dixmier1953}. Our reference is \cite{PX2003}.

Let $\mathcal{M}$ be a semifinite von Neumann algebra equipped with a n.s.f. trace $\tau$. Denote by $\mathcal{M}_+$ the positive cone of $\mathcal{M}$ and by $\mathcal{S}_+$ the set of all $x \in \mathcal{M}_+$ such that $\tau [\mathrm{supp} (x)] < \infty$, where $\mathrm{supp}(x)$ denotes the support projection of $x$. Let $\mathcal{S}$ be the linear span of $\mathcal{S}_+$. Then $\mathcal{S}$ is a weak*-dense *-subalgebra of $\mathcal{M}$. Given $0<p<\infty$, we define
$$\Vert x\Vert_p:=[\tau(|x|^p)]^{\frac{1}{p}},~~x\in\mathcal{S},$$
where $|x|=(x^*x)^{\frac{1}{2}}$ is the modulus of $x$. One can show that $\|\cdot \|_p$ is a norm on $\mathcal{S}$ if $1 \leq p < \infty,$ and a quasi-norm if $0< p < 1.$ In particular, we have for $p\ge 1$
\begin{equation}\label{eq:minkowski inequality p>1}
\|x+y\|_p\le \|x\|_p+\|y\|_p,~~x,y\in\mathcal{S};
\end{equation}
and for $0<p<1$:
\begin{equation}\label{eq:minkowski inequality p<1}
\|x+y\|^p_p\le \|x\|^p_p+\|y\|^p_p,~~x,y\in\mathcal{S}.
\end{equation}
The completion of $( \mathcal{S}, \|\cdot \|_p )$ is denoted by $\Lambda_p (\mathcal{M}, \tau)$, or simply $\Lambda_p (\mathcal{M})$. It is called \emph{tracial noncommutative $L_p$-space associated with $(\mathcal{M},\tau)$}. As usual, we set $\Lambda_\infty(\mathcal{M},\tau)=\mathcal{M}$ equipped with the operator norm.

The elements of $\Lambda_p(\mathcal{M})$ can be viewed as closed densely defined operators on $H$ ($H$ being the Hilbert space on which $\mathcal{M}$ acts). A linear closed operator $x$ is said to be \textit{affiliated with} $\mathcal{M}$ if it commutes with all unitary elements in $\mathcal{M}'$, i.e., $xu=ux$ for any unitary $u\in\mathcal{M}'$, where $\mathcal{M}'$ denotes the commutant of $\mathcal{M}$. Note that $x$ can be unbounded on $H$. An operator $x$ affiliated with $\mathcal{M}$ is said to be \textit{measurable with respect to $(\mathcal{M},\tau)$}, or simply \textit{measurable} if for any $\varepsilon>0$, there exists a projection $e\in\mathcal{M}$ such that
$$e(H)\subset\mathcal{D}(x) \text{ and }  \tau(e^\perp)\leq\varepsilon,$$
where $e^\perp=1-e$ and $\mathcal{D}(x)$ denotes the domain of $x$. We denote by $L_0(\mathcal{M},\tau)$, or simply $L_0(\mathcal{M})$ the family of measurable operators. For such an operator $x$, we define the \textit{distribution function} of $x$ as
$$\lambda_s(x):=\tau[\chi_{(s,\infty)}(|x|)],~~s\geq0,$$
where $\chi_{(s,\infty)}(|x|)$ is the spectral projection of $|x|$ corresponding to the interval $(s,\infty)$, and define the \textit{generalized singular numbers} of $x$ as
$$\mu_t(x):=\inf\{s>0:\lambda_s(x)< t\},~~t\geq 0.$$
It is easy to check that 
$$\|x\|=\lim\limits_{t\to 0^+}\mu_t(x),$$
and for $0<p<\infty$
\begin{equation}\label{eq:p norm in term of singular numbers}
\|x\|^p_p=\int_{0}^{\infty}\mu_t(x)^p dt.
\end{equation}

\subsection{Haagerup noncommutative $L_p$-spaces}
For Haagerup noncommutative $L_p$-spaces we refer to \cite{T1981}. The Haagerup noncommutative $L_p$-space is defined on a general ($\sigma$-finite) von Neumann algebra $\mathcal{M}$ associated with a n.f. state $\phi$. In fact, the Haagerup noncommutative $L_p$-space can be defined for any n.s.f. weight, but we will consider only the state case.

Let $\mathcal{R}$ denote the crossed product $\mathcal{M} \rtimes_{\sigma^\phi} \mathbb{R},$ where $\{ \sigma_t^\phi\}_{t\in \mathbb{R}}$ is the \emph{modular automorphism group} associated with $\phi$ on $\mathcal{M}$ (for general modular theory see \cite{Takesaki}). 

Recall that if $\mathcal{M}$ acts on a Hilbert space $H$, $\mathcal{R}$ is the von Neumann algebra acting on $L_2 (\mathbb{R}, H )$ generated by the operators $\pi (x), x \in \mathcal{M},$ and $\lambda (s), s \in \mathbb{R},$ such that: for every $\xi \in L_2 (\mathbb{R}, H )$ and $t \in \mathbb{R},$
$$  \pi (x) (\xi) (t) = \sigma_{-t} (x)  \xi (t), ~~  \lambda(s) (\xi) (t) = \xi (t-s).$$
$\pi$ is a normal faithful representation of $\mathcal{M}$ on $L_{2}(\mathbb{R};H)$, so we can identify $\pi(\mathcal{M})$ with $\mathcal{M}$. There is a dual action $\{\hat{\sigma}_t^\phi \}_{t \in \mathbb{R}}$ of $\mathcal{R}$ uniquely given by
$$ \hat{\sigma}_t^\phi (x)=x,~~ \hat{\sigma}_t^\phi (\lambda(s)) = e^{-ist} \lambda(s),$$
for $x \in \mathcal{M}$ and $s, t \in \mathbb{R}.$

It is known that $\mathcal{R}$ is semifinite and there is a canonical n.s.f. trace $\tau$ on $\mathcal{R}$ satisfying
$$\tau \circ \hat{\sigma}_t^\phi = e^{-t} \tau, \; t \in \mathbb{R},$$

\noindent Any normal positive functional $\psi \in \mathcal{M}_*^+$ induces a dual normal semifinite weight $\hat{\psi}$ on $\mathcal{R}$ defined by
\begin{equation}
\hat{\psi} (x) = \psi \left[ \int_{-\infty}^\infty \hat{\sigma}_t^\phi (x) dt \right], \; x \in \mathcal{R}^+.
\end{equation}

\noindent The dual weight $\hat{\psi}$ has a \emph{Radon-Nikodym derivative} with respect to $\tau$, denoted by $h_\psi$:
$$\tau(h_\psi x) = \hat{\psi} (x), \; x \in \mathcal{R}^+.$$
Here $\tau(h_\psi \cdot)$ is understood as $\tau(h_\psi^{\frac{1}{2}}\cdot h_\psi^{\frac{1}{2}})$. In particular, the dual weight $\hat{\phi}$ of our distinguished n.s.f. state $\phi$ has a Radon-Nikodym derivative $h_\phi$ with respect to $\tau$. We will call $h_\phi$ the \emph{density operator} of $\phi$.

Let $L_0(\mathcal{R}, \tau)$ denote the topological $*$-algebra of all operators on $L_2(\mathbb{R}, H)$ which are measurable with respect to $(\mathcal{R}, \tau)$. Then for $0< p \leq \infty$, the Haagerup $L_p$-space is defined as
\begin{equation}\label{eq:defn of Haagerup Lp space}
L_p(\mathcal{M}, \phi) = \{ x \in L_0(\mathcal{R}, \tau): \; \hat{\sigma}_t^\phi (x) = e^{-t/p} x, \; \forall t \in \mathbb{R}\}.
\end{equation}

Recall that for any $\psi \in \mathcal{M}_*^+$, we have $h_\psi \in L_0(\mathcal{R}, \tau)$ and
$$\hat{\sigma}_t^\phi (h_\psi) = e^{-t} h_\psi, \; \forall t \in \mathbb{R}.$$
Thus $h_\psi \in L_1(\mathcal{M}, \phi)_+.$ This correspondence between $\mathcal{M}_*^+$ and $L_1(\mathcal{M}, \phi)_+$ extends to a bijection between $\mathcal{M}_*$ and $L_1(\mathcal{M}, \phi).$ More precisely, for any $\psi \in \mathcal{M}_*,$ if $\psi = u |\psi|$ is its polar decomposition, then we have
$$h_\psi = u h_{|\psi|}.$$
Thus we may define a norm on $L_1 (\mathcal{M}, \phi)$ by
$$\|h_\psi\|_1 := \|\psi\|_{\mathcal{M}_*} = |\psi| (1), \; \psi \in \mathcal{M}_*.$$
Hence $L_1 (\mathcal{M}, \phi) = \mathcal{M}_*$ isometrically. Due to this correspondence, there is a linear functional on $L_1(\mathcal{M}, \phi)$, denoted by $\tr$, given through
$$\tr (x) =  \psi_x (1), \; x \in L_1(\mathcal{M}, \phi),$$
where $\psi_x$ is the unique normal functional associated with $x$ by the identification between $\mathcal{M}_*$ and $L_1(\mathcal{M}, \phi)$ we discussed earlier.

Now for $x \in L_p (\mathcal{M}, \phi), 0< p< \infty,$ we can define
\begin{equation}\label{eq:defn of haagerup lp norm}
\|x\|_p := \||x|^p\|_1^{\frac{1}{p}} = [\tr (|x|^p)]^{\frac{1}{p}}.
\end{equation}
Then $\|\cdot\|_p$ is a norm (quasi-norm if $p<1$) on $L_p(\mathcal{M}, \phi).$
For Haagerup $L_p$-spaces we have the following H\"{o}lder's inequality which will be used frequently: for every $p, q, r > 0$ with $\frac{1}{r} = \frac{1}{p} + \frac{1}{q},$ we have
\begin{equation}\label{eq:holder}
\| x y \|_r \leq \|x\|_p \| y\|_q, \; x \in L_p(\mathcal{M}, \phi), y \in L_q(\mathcal{M}, \phi).
\end{equation}
Moreover, from \eqref{eq:defn of haagerup lp norm} it follows directly that 
\begin{equation}\label{eq:p norm and 2p norm}
\|x\|^{2}_{2p}=\|x^*x\|_{p}=\|xx^*\|_{p},~~x\in L_{2p}(\mathcal{M},\phi).
\end{equation}
Note that from \eqref{eq:defn of Haagerup Lp space}, for different $p_0$ and $p_1$, the intersection of $L_{p_0}(\mathcal{M}, \phi)$ and $L_{p_1}(\mathcal{M}, \phi)$ is trivial.

Let $\mathcal{M}_a$ be the family of \emph{analytic elements} in $\mathcal{M}$. Recall that for any $x \in \mathcal{M}$, it belongs to $\mathcal{M}_a$ iff $t \mapsto \sigma_t^\phi (x)$ extends to an analytic function from $\mathbb{C}$ to $\mathcal{M}$. We will use the following technical lemma proved by Junge and Xu \cite{JX2003}.
\begin{lemma}\label{lem:dense}\cite[Lemma 1.1]{JX2003}.
	Let $0< p < \infty, 0 \leq \eta \leq 1.$ Then:
	\begin{enumerate}[(i)] 
		\item $h_\phi^{\frac{1-\eta}{p}} \mathcal{M}_a h_\phi^{\frac{\eta}{p}} = \mathcal{M}_a h_\phi^{\frac{1}{p}};$
		\item $\mathcal{M}_a h_\phi^{\frac{1}{p}}$ is dense in $L_p (\mathcal{M}, \phi).$
	\end{enumerate}
\end{lemma}
We note that here $h_\phi^{\frac{1-\eta}{p}} \mathcal{M}_a h_\phi^{\frac{\eta}{p}} = \mathcal{M}_a h_\phi^{\frac{1}{p}}$ follows from
\begin{equation}\label{eq:proof of lem of Junge-Xu}
xh_\phi^{\frac{1}{p}} = h_\phi^{\frac{1-\eta}{p}} \sigma^\phi_{ \frac{i(1-\eta)}{p}} (x) h_\phi^{\frac{\eta}{p}}, ~~x\in \mathcal{M}_a.
\end{equation}

Now we mention two basic facts about Haagerup $L_p$-spaces. The first is, when $\mathcal{M}$ is a semifinite von Neumann algebra associated with a n.f. tracial state $\tau$, Haagerup $L_p$-space $L_p(\mathcal{M},\tau)$ and tracial $L_p$-space $\Lambda_p(\mathcal{M},\tau)$ can be isometrically identified \cite[Chapter II]{T1981}. In fact, in this case, $\mathcal{M} \rtimes_{\sigma^\tau} \mathbb{R} = \mathcal{M} \overline{\otimes} L_\infty(\mathbb{R})$ and the density operator of $\tau$ is: $h_\tau = id \otimes \exp (\cdot).$ For each $x\in L_p (\mathcal{M}, \tau)$ there exists unique $x'\in \Lambda_p(\mathcal{M}, \tau)$ such that $x=x'\otimes \exp \left( \cdot/p \right)$. Moreover, $\|x\|_{L_p(\mathcal{M},\tau)}=\|x'\|_{\Lambda_p(\mathcal{M},\tau)}$. Thus the map $x\mapsto x'$ yields an isometry between $L_p(\mathcal{M},\tau)$ and $\Lambda_p(\mathcal{M},\tau)$.

The second well-known fact is that Haagerup $L_p$-spaces $L_p(\mathcal{M}, \phi)$ are independent of the choice of state $\phi$ up to isometry \cite[Chapter II]{T1981}. Namely, suppose that $\varphi_0,\varphi_1$ are two n.f. states on a von Neumann algebra $\mathcal{M}$. Recall that $\mathcal{R}_{i}:=\mathcal{M}\rtimes_{\sigma^{\varphi_{i}}}\mathbb{R}$ is generated by $\pi_{i}(x),~x\in\mathcal{M}$ and $\lambda(s),~s\in\mathbb{R}$, where
\[
\pi_{i}(x)\xi(t)=\sigma^{\varphi_{i}}_{-t}(x)\xi(t),~\lambda(s)\xi(t)=\xi(t-s),~~\xi\in L_{2}(\mathbb{R},H),~s,t\in\mathbb{R},
\]
for $i=0,1$. Then there is an isometry $\kappa:L_p(\mathcal{M}, \varphi_0)\to L_p(\mathcal{M}, \varphi_1)$ such that \cite[Theorem 37]{T1981}
$$\kappa[\pi_0(x)]=\pi_1(x),~~\kappa[\lambda(t)]=\pi_0(u_t^*)\lambda(t),~~t\in\mathbb{R},x\in\mathcal{M},$$ 
where $u_t=(h_{\varphi_{1}}:h_{\varphi_{0}})_t$ is the \emph{Radon-Nikodym cocycle} of $\varphi_{1}$ relative to $\varphi_{0}$:
\[
u_{s+t}=u_{s}\sigma^{\varphi_{0}}_{t}(u_{t}),~~\sigma^{\varphi_{1}}_{t}(x)=u_{t}\sigma^{\varphi_{0}}_{t}(x)u^{*}_{t},
\]
where $s,t\in\mathbb{R}$, $x\in\mathcal{M}$, and $h_{\varphi_i}$ is the density operator of $\varphi_i$, $i=0,1$. In particular, if $\varphi_1=\varphi_0(h\cdot)$, then $u_t=h^{it}$. Thus we have $\kappa[\pi_0(h) h_{\varphi_0}]=h_{\varphi_1}$.

Then combining the above two facts, we obtain 
\begin{lemma}\label{lem:isometry}
	Let $\mathcal{M}$ be a von Neumann algebra associated with a n.f. tracial state $\tau$. Let $\phi$ and $\psi$ be two n.f. states on $\mathcal{M}$ defined by:
	\begin{equation*}
	\psi (x): = \tau (\rho x),~\phi(x):=\tau(\sigma x),~ x \in \mathcal{M},
	\end{equation*}
	where $\rho,\sigma \geq 0$ and $\tau (\rho)=\tau(\sigma)=1.$ Then there exist an isometry $\theta:L_p(\mathcal{M}, \tau) \to L_p(\mathcal{M}, \phi)$ and an isometry $\theta':\Lambda_p(\mathcal{M}, \tau) \to L_p(\mathcal{M}, \phi)$ such that
	\begin{equation}\label{eq:eq 1 in lem on isometry}
	h_{\psi}
	=\theta[h_{\tau}\pi_{\sigma^{\tau}}(\rho)]
	=\theta'(\rho),
	\end{equation}
	where $h_\psi$ is the Radon-Nikodym derivative of $\psi$ with respect to the canonical trace $\tau^\phi$ on $\mathcal{M}\rtimes_{\sigma^\phi}\mathbb{R}$ and $h_\tau$ is the density operator of $\tau$. In particular, we have
	\begin{equation}\label{eq:eq 2 in lem on isometry}
	h_{\phi}
	=\theta[h_{\tau}\pi_{\sigma^{\tau}}(\sigma)]
	=\theta'(\sigma).
	\end{equation}
\end{lemma}

\begin{proof}
	The equation \eqref{eq:eq 2 in lem on isometry} is a direct consequence of two facts as above. Note that $\psi=\phi(x^*\cdot x)$, with $x=\rho^{\frac{1}{2}}\sigma^{-\frac{1}{2}}$. Then \eqref{eq:eq 1 in lem on isometry} follows immediately from \eqref{eq:eq 2 in lem on isometry} and the fact \cite[Chapter II, Proposition 4]{T1981} that $h_{\phi(x^*\cdot x)}=x h_\phi x^*$.
\end{proof}

Using the above lemma, we can show that the noncommtuative $L_p$-space $L_{p, \sigma}(H)$ is a Haagerup $L_p$-space. 
\begin{definition}
	Let $\mathcal{M}$ be a $\sigma$-finite von Neumann algebra equipped with a n.f. state $\phi.$ For $0< p < \infty,$ define a (quasi) norm $\|\cdot\|_{p, \phi}$ on $\mathcal{M}$ by
	$$\|x\|_{p, \phi} : = \left\| h_\phi^{\frac{1}{2p}} x h_\phi^{\frac{1}{2p}} \right\|_{L_p(\mathcal{M}, \phi)},  \; x\in \mathcal{M}.$$
\end{definition}

Recall \cite{Beigi2013} that for a fixed positive operator $\sigma \in B(H),$ the space $L_{p, \sigma} (H)$ is defined as the completion of $(B(H), \|\cdot\|_{p, \sigma}),$ where the (quasi) norm $\|\cdot\|_{p, \sigma}$ is defined by
$$\|x\|_{p, \sigma} =  \left\| \sigma^{\frac{1}{2p}} x \sigma^{\frac{1}{2p}} \right\|_p, x \in B(H)$$
for $0 < p < \infty.$ Here $\|\cdot\|_p$ is the Schatten $p$-norm on $B(H).$

\begin{proposition}\label{prop:fit}
	Let $\mathcal{M}$ be a finite von Neumann algebra with a n.f. tracial state $\tau$, and $\phi$ be a n.f. state on $\mathcal{M}$ defined by $\phi=\tau(\sigma\cdot)$, where $\sigma\ge 0$ and $\tau(\sigma)=1$. Then
	\begin{equation}
	\|x\|_{p, \phi}=\left[\tau \left(\left|\sigma^{\frac{1}{2p}} x \sigma^{\frac{1}{2p}}\right|^p \right)\right]^{\frac{1}{p}},~ x \in \mathcal{M}
	\end{equation}
	for $0 < p < \infty.$
\end{proposition}

\begin{proof}
	Using Lemma \ref{lem:isometry}, we have
	\begin{equation*}
	\begin{split}
	\left\| h_\phi^{\frac{1}{2p}} x h_\phi^{\frac{1}{2p}} \right\|^p_{L_p(\mathcal{M}, \phi)} 
	& = \left\|\theta' \left( \sigma^{\frac{1}{2p}}  x \sigma^{\frac{1}{2p}} \right) \right\|^p_{L_p(\mathcal{M}, \phi)} \\
	& = \left\| \sigma^{\frac{1}{2p}}  x \sigma^{\frac{1}{2p}} \right\|^p_{\Lambda_p(\mathcal{M}, \tau)}\\
	& = \tau \left( \left| \sigma^{\frac{1}{2p}} x \sigma^{\frac{1}{2p}} \right|^p \right),
	\end{split}
	\end{equation*}
	where $\theta'$ is the isometry from $\Lambda_p(\mathcal{M}, \tau)$ to $L_p(\mathcal{M}, \phi)$.
\end{proof}

\subsection{Haagerup's reduction theorem}
The idea of the Haagerup reduction theorem is to approximate a general von Neumann algebra by finite ones. It was firstly proposed by Haagerup and published in \cite{HJX2010} after amending more details and applications to noncommutative martingale inequalities and maximal inequalities. In this paper, we will use the Haagerup reduction theorem for noncommutative $L_p$-spaces, which briefly states that a Haagerup $L_p$-space associated with a general von Neumann algebra can be approximated by a sequence of tracial $L_p$-spaces.

Let $G$ denotes the subgroup $\bigcup_{n \geq 1} 2^{-n} \mathbb{Z}$ of $\mathbb{R}.$ Let $\mathcal{M}$ be a $\sigma$-finite von Neumann algebra associated with a normal faithful state $\phi.$ Denote $\mathcal{R}: = \mathcal{M} \rtimes_{\sigma^\phi} G.$ Then we have the following reduction theorem.

\begin{theorem}\label{thm:reduction}\cite[Theorem 3.1]{HJX2010}. For $0< p < \infty,$ let $L_p(\mathcal{M}, \phi)$ be the Haagerup noncommutative $L_p$-space. Then there exist a sequence $\{\mathcal{R}_n\}_{n \geq 1}$ of finite von Neumann algebras, each equipped with a normal faithful finite trace $\tau_n,$ and for each $n \geq 1$ an isometric embedding $\theta_n: \Lambda_p(\mathcal{R}_n, \tau_n) \to L_p (\mathcal{R}, \hat{\phi})$ such that
	\begin{enumerate}[(i)] 
		\item the sequence $\{ \theta_n [\Lambda_p (\mathcal{R}_n, \tau_n) ] \}_{n\geq 1}$ is increasing;
		\item $\bigcup_{n\geq1}  \theta_n [\Lambda_p (\mathcal{R}_n, \tau_n) ]$ is dense in  $L_p (\mathcal{R}, \hat{\phi})$;
		\item $L_p(\mathcal{M}, \phi)$ is isometric to a subspace $Y_p$ of $L_p (\mathcal{R}, \hat{\phi});$
		\item $Y_p$ and all $\theta_n [\Lambda_p (\mathcal{R}_n, \tau_n) ]$ are 1-complemented in $L_p (\mathcal{R}, \hat{\phi})$ for $1\leq p < \infty.$
	\end{enumerate}
	Here $\Lambda_p (\mathcal{R}_n, \tau_n)$ is the tracial noncommutative $L_p$-space associated with $(\mathcal{R}_n, \tau_n).$
\end{theorem}

\subsection{Complex interpolation} 
We refer to \cite{BL1976} for more information in complex interpolation theory. Let $S$ denote the strip $\{ z\in \mathbb{C}: 0 < Re z < 1 \}$, and $\bar{S}$ its closure. Let $A(S)$ be the set of all bounded functions which are analytic in $S$ and continuous in $\bar{S}.$ Let $(X_0, X_1)$ be an interpolation couple of (quasi) Banach spaces. Set
$$\mathcal{F} (X_0, X_1) := \left\{f: f(z) = \sum_{k=1}^m f_k(z) x_k, x_k \in X_0 \cap X_1, f_k \in A(S),m\in\mathbb{N} \right\}.$$
Equipped with the norm
$$\|f\|_{\mathcal{F} (X_0, X_1)} = \sup_{t \in \mathbb{R}} \left\{ \|f(it)\|_{X_0}, \|f(1+it)\|_{X_1} \right\},$$
$\mathcal{F} (X_0, X_1)$ becomes a (quasi) Banach space. For $0 < \theta <1$ we define the complex interpolation (quasi) norm $\|\cdot\|_\theta$ on $X_0 \cap X_1$ as follows
\begin{equation}
\|x\|_{\theta} = \inf \left\{ \|f\|_{\mathcal{F} (X_0, X_1)}: f(\theta) =x, f\in \mathcal{F} (X_0, X_1) \right\}, \; x \in X_0 \cap X_1.
\end{equation}
Then the complex interpolation space of $X_0$ and $X_1$ is defined as the completion of $(X_0 \cap X_1, \|\cdot\|_\theta)$, denoted by $[X_0, X_1]_\theta$.

For the tracial noncommutative $L_p$-space $\Lambda_p(\mathcal{M}, \tau)$, it is well-known that
\begin{equation}
[\Lambda_{p_0} (\mathcal{M}, \tau), \Lambda_{p_1}(\mathcal{M}, \tau)]_\theta = \Lambda_{p_\theta}(\mathcal{M}, \tau),
\end{equation}
where $1\leq p_0 < p_1 \le \infty,$ and $\frac{1}{p_\theta} = \frac{1-\theta}{p_0} + \frac{\theta}{p_1}$ (see \cite{PX2003}). By showing a type of reversed H\"{o}lder's inequality (see Lemma \ref{lem:decomposition}), Xu \cite{X1990} proved the above complex interpolation result for the quasi-Banach space $\Lambda_p(\mathcal{M}, \tau)$ :
\begin{equation}\label{eq:Xu-result}
[\Lambda_{p_0} (\mathcal{M}, \tau), \Lambda_{p_1}(\mathcal{M}, \tau)]_\theta = \Lambda_{p_\theta}(\mathcal{M}, \tau),
\end{equation}
where $0< p_0 < p_1 \leq \infty$ and $\frac{1}{p_\theta} = \frac{1-\theta}{p_0} + \frac{\theta}{p_1}$.

\subsection{Compatibility of Haagerup $L_p$-spaces}
As we have mentioned in the introduction, the compatibility of two (quasi) Haagerup noncommutative $L_p$-spaces is due to Kosaki's construction. More precisely, we have the following proposition: 

\begin{proposition}\label{prop:embed1}
	Let $\mathcal{M}$ be a $\sigma$-finite von Neumann algebra with a n.f. state $\phi.$ For $0< p_0 < p_1\leq \infty,$ and $x\in \mathcal{M},$ we have 
	$$\| x \|_{p_0, \phi}
	\le \|x\|_{p_1, \phi}.$$
	%Thus $(\mathcal{L}_{p_0} (\mathcal{M}), \mathcal{L}_{p_1}(\mathcal{M}))$ is an interpolation couple.
\end{proposition}

\begin{proof}
	For any $x \in \mathcal{M},$ $h_\phi^{\frac{1}{2p_0}} x h_\phi^{\frac{1}{2p_0}}$ can be decomposed as
	$$h_\phi^{\frac{1}{2p_0}} x h_\phi^{\frac{1}{2p_0}} 
	=h_\phi^{\frac{1}{q}}   h_\phi^{\frac{1}{2p_1}} x h_\phi^{\frac{1}{2p_1}} h_\phi^{\frac{1}{q}} 
	=h_\phi^{\frac{1}{q}}  y h_\phi^{ \frac{1}{q}},$$
	where $\frac{1}{q} +  \frac{1}{2p_1} = \frac{1}{2p_0}$ and $y = h_\phi^{\frac{1}{2p_1}} x h_\phi^{\frac{1}{2p_1}}$. By H\"{o}lder's inequality \eqref{eq:holder}, we have
	\begin{equation*}
	\begin{split}
	\left\|h_\phi^{\frac{1}{2p_0}} x h_\phi^{\frac{1}{2p_0}} \right\|_{L_{p_0}(\mathcal{M}, \phi)} 
	&\leq \|y\|_{L_{p_1}(\mathcal{M}, \phi)} \left\| h_\phi^{\frac{1}{q}} \right\|_{L_q(\mathcal{M}, \phi)}^2\\
	& =  \|y\|_{L_{p_1}(\mathcal{M}, \phi)} \left\| h_\phi \right\|_{L_1(\mathcal{M}, \phi)}^{\frac{2}{q}} 
	= \|y\|_{L_{p_1}(\mathcal{M}, \phi)}.\qedhere
	\end{split}
	\end{equation*}
\end{proof}

Due to the above proposition, for $0< p_0 < p_1\leq \infty$, $L_{p_1} (\mathcal{M}, \phi)$ can be embedded into $L_{p_0} (\mathcal{M}, \phi)$ via the map $h_\phi^{\frac{1}{2p_1}}xh_\phi^{\frac{1}{2p_1}} \mapsto  h_\phi^{\frac{1}{2p_0}}xh_\phi^{\frac{1}{2p_0}}.$ With this embedding, $(L_{p_0} (\mathcal{M}, \phi), L_{p_1}(\mathcal{M}, \phi))$ forms an interpolation couple.

%%%%%%%%%%%%%%%%%%%%%%%%

\section{Hadamard three lines theorem for (quasi) noncommutative $L_p$-spaces}

\subsection{The tracial case.}

In this subsection, we denote by $\mathcal{M}$ a semifinite von Neumann algebra associated with a n.s.f. trace $\tau.$ For $\theta \in [0, 1],$ and $0< p_0 \leq p_1 \leq \infty,$ define $p_\theta$ as follows:
\begin{equation}\label{eq:conjugate-index}
\frac{1}{p_\theta} = \frac{1-\theta}{p_0} + \frac{\theta} {p_1}.
\end{equation}

Recall that $\bar{S}= \{ z\in \mathbb{C}: 0 \leq  Re z \leq1 \}$. Denote by $AF(X)$ the set of all bounded functions from $\bar{S}$ to $X$ which are analytic in $S$ and continuous on $\bar{S}$. The following reversed H\"older's inequality is due to Xu.

\begin{lemma}\cite[Lemma 2.2]{X1990}\label{lem:decomposition}.
	Let $\mathcal{M}$ be a von Neumann algebra associated with a n.s.f. trace $\tau$ such that $\tau(1) < \infty$. Fix $0<\lambda<1$. Then for any $f \in AF (\mathcal{M})$ and any $\epsilon >0,$ there exist $g, h \in AF(\mathcal{M})$ such that $f= gh$ and for all $z\in\partial S,t\ge 0$ we have
	\begin{equation*}
	\left\{
	\begin{aligned}
	\mu_t(g(z)) & \leq \mu_t(f(z))^{1-\lambda} + \epsilon,\\
	\mu_t(h(z)) & \leq \mu_t(f(z))^{\lambda} + \epsilon.
	\end{aligned}
	\right.
	\end{equation*}
\end{lemma}

Remark here that our statement is slightly different from Xu's original form. Indeed, in \cite{X1990} $\mathcal{M}$-valued analytic functions in $AF(\mathcal{M})$ are defined on the interval $[0,2\pi]$. But they are equivalent, since one can identify $[0,2\pi]$ with the unit circle and then apply Riemann mapping theorem. As a direct consequence of \eqref{eq:p norm in term of singular numbers} and Lemma \ref{lem:decomposition} we have

\begin{corollary}\label{cor:cor of xu's result}
	Let $(\mathcal{M},\tau)$ and $\lambda$ be as above. Then for any $f \in AF (\mathcal{M})$ and any $\epsilon,r >0,$ there exist $g, h \in AF(\mathcal{M})$ such that $f= gh$ and for all $z\in\partial S$ we have
	\begin{equation*}
	\left\{
	\begin{aligned}
	\|g(z)\|_{\frac{r}{1-\lambda}} & \leq \|f(z)\|_{r}^{1-\lambda} + \epsilon,\\
	\|h(z)\|_{\frac{r}{\lambda}} & \leq \|f(z)\|_{r}^{\lambda} + \epsilon.
	\end{aligned}
	\right.
	\end{equation*}
\end{corollary}

\begin{remark}
	The above corollary is a special case of the Riesz factorization theorem, which was proved in the general case of subdiagonal algebras by Blecher-Labuschagne \cite{BL2006} and Bekjan-Xu \cite{BX2007}.
\end{remark}

\begin{proposition}\label{prop:three-line}
	Let $\mathcal{M}$ be a semifinite von Neumann algebra associated with a n.s.f. trace $\tau$. Let $G \in AF( \mathcal{M})$. Assume that  $0 < p_0 < p_1 \leq \infty,$ and for $\theta \in [0, 1]$ define $p_\theta$ as in the equation \eqref{eq:conjugate-index}.
	For $k =0,1,$ define
	$$M_k := \sup_{t \in \mathbb{R}} \|G(k+it)\|_{p_k}.$$
	Then we have
	\begin{equation*}
	\|G(\theta) \|_{p_\theta} \leq  M_0^{1-\theta} \; M_1^{\theta}.
	\end{equation*}
\end{proposition}

\begin{proof}
	The proof of Banach spaces case $1\le p_0<p_1\le\infty$ is a standard argument, using Hadamard three lines theorem for analytic functions. We refer to \cite[Theorem 2]{Beigi2013} and \cite[Theorem 2.10]{Jencova2018} for the proof of (not necessarily tracial) state case for $1\le p_0<p_1\le\infty$. For the proof for quasi-Banach spaces case $0< p_0<p_1\le\infty$, choose $n\in\mathbb{N}$ such that $n p_0\ge 1$. Then for any $\epsilon>0$, by Corollary \ref{cor:cor of xu's result} there exist $G_1,\cdots, G_n\in AF(\mathcal{M})$ such that $G=\prod_{j=1}^{n}G_j$ and for $1\le j\le n$
	\begin{equation}\label{ineq:factorization for n}
	\sup_{t \in \mathbb{R}} \|G_j(k+it)\|_{n p_k}\le \sup_{t \in \mathbb{R}} \|G(k+it)\|_{ p_k}^{\frac{1}{n}}+\epsilon,~k=0,1.
	\end{equation}
	
	Hence by H\"older's inequality and the desired result for Banach spaces case, we have
	\begin{equation*}
	\|G(\theta) \|_{p_\theta} 
    \leq \prod_{j=1}^{n}\|G_j(\theta)\|_{np_\theta}
	\leq  \prod_{j=1}^{n}\left(\sup_{t \in \mathbb{R}} \|G_j(it)\|_{n p_0}\right)^{1-\theta}\left(\sup_{t \in \mathbb{R}} \|G_j(1+it)\|_{n p_1}\right)^{\theta}.
	\end{equation*}
	Applying \eqref{ineq:factorization for n}, we obtain
	\begin{equation*}
	\begin{split}
	\|G(\theta) \|_{p_\theta} 
	&\leq\prod_{j=1}^{n}\left(\sup_{t \in \mathbb{R}} \|G(it)\|_{p_0}^{\frac{1}{n}}+\epsilon\right)^{1-\theta}\left(\sup_{t \in \mathbb{R}} \|G(1+it)\|_{p_1}^{\frac{1}{n}}+\epsilon\right)^{\theta}\\
	&\leq \left(\sup_{t \in \mathbb{R}} \|G(it)\|_{p_0}\right)^{1-\theta} \left(\sup_{t \in \mathbb{R}} \|G(1+it)\|_{ p_1}\right)^{\theta}+O(\epsilon).
	\end{split}
	\end{equation*}
	Letting $\epsilon\to 0^+$, we deduce our desired result.
\end{proof}

\subsection{The Haagerup noncommutative $L_p$-spaces case.}

\begin{lemma}\label{lem:three-line-Haagerup}
	Let $\mathcal{M}$ be a semifinite von Neumann algebra with a n.f. trace $\tau.$ Let $G\in AF(\mathcal{M})$. Assume that  $0 < p_0 < p_1 \leq  \infty, $ and for $\theta \in [0, 1]$ define $p_\theta$ as in the equation \eqref{eq:conjugate-index}.
	For $k =0,1$ define
	$$M_k:= \sup_{t \in \mathbb{R}} \|G(k+it)\|_{p_k ,\tau}.$$
	Then we have
	\begin{equation*}
	\|G(\theta) \|_{p_\theta, \tau} \leq  M_0^{1-\theta} \; M_1^{\theta}.
	\end{equation*}
\end{lemma}

\begin{proof}
	It is obvious that  $$\|G(z)\|_{p_z, \tau} = \left\|h_\tau^{\frac{1}{2p_z}} G(z) h_\tau^{\frac{1}{2p_z}} \right\|_{L_{p_z}(\mathcal{M}, \tau)} = \|G(z)\|_{\Lambda_{p_z}(\mathcal{M}, \tau)},$$
	where $\frac{1}{p_z} = \frac{1-Re(z)}{p_0} + \frac{Re(z)}{p_1}.$ Then our result can be deduced from Proposition \ref{prop:three-line}.
\end{proof}

The following proposition extends Beigi's result \cite[Theorem 2]{Beigi2013} (see also Jen\v{c}ov\'a's result \cite[Theorem 2.10]{Jencova2018}) to the quasi-Banach space case.

\begin{proposition}\label{prop:three-line-Haagerup}
	Let $\mathcal{M}$ be a ($\sigma$-finite) von Neumann algebra with a n.s.f. state $\phi.$ Let $G\in AF(\mathcal{M})$. Assume that  $0 < p_0 < p_1 \leq  \infty, $ and for $\theta \in [0, 1]$ define $p_\theta$ as in the equation \eqref{eq:conjugate-index}.
	For $k =0,1$ define
	$$M_k := \sup_{t \in \mathbb{R}} \|G(k+it)\|_{p_k ,\phi}.$$
	Then we have
	\begin{equation*}
	\|G(\theta) \|_{p_\theta, \phi} \leq  M_0^{1-\theta} \; M_1^{\theta}.
	\end{equation*}
\end{proposition}

\begin{proof}
	The key idea is to use the Haagerup reduction theorem. Recall that $\mathcal{R} =  \mathcal{M} \rtimes_{\sigma^\phi}\cup_{n\ge 1} 2^{-n} \mathbb{Z}$, and $L_p(\mathcal{M}, \phi)$ can be isometrically embedded into $L_p(\mathcal{R}, \hat{\phi})$ for any $0< p< \infty.$ Hence it suffices to show that for $G \in AF (\mathcal{R})$
	\begin{equation*}
	\|G(\theta) \|_{p_\theta, \hat{\phi}} \leq  M_0^{1-\theta} \; M_1^{\theta},
	\end{equation*}
	where
	$M_k:= \sup_{t \in \mathbb{R}} \|G(k+it)\|_{p_k, \hat{\phi}},~k=0,1$.
	
	According to Theorem \ref{thm:reduction}, $\bigcup_n L_p(\mathcal{R}_n, \tau_n)$ is dense in $L_p (\mathcal{R}, \hat{\phi})$. Hence without loss of generality, we can assume that there exists an integer $N,$ such that $x=G(\theta) \in \mathcal{R},$ and
	$M_k \leq 1$ for $k=0,1.$ Suppose that
	$$G(z) = \sum_{m} G_m(z) x_m,$$
	where $G_m \in A(S)$ and $x_m \in \mathcal{R}_N.$ Then for $\epsilon >0,$ there exists a sequence $\{y_m\} \subset \mathcal{R}_{N'}$ such that for each $m$
	\begin{equation}\label{eq:approximation}
	\|x_m - y_m\|_{p_1, \tau_{N'}} < \frac{\epsilon}{2^m K_m},
	\end{equation}
	where $N' > N$ and $K_m:=\sup\{|G_m(z)|:z\in\overline{S}\}<\infty$.
	Hence
	$$G = \sum_m G_m y_m + \sum_m G_m (x_m - y_m)=:g+ h,$$
	where $g(z) = \sum_m G_m(z) y_m$ and $h(z) = \sum_m G_m(z) (x_m - y_m).$ Then if $p_0\ge 1$, using the inequality \eqref{eq:minkowski inequality p>1} we have
	\begin{equation*}
	\begin{split}
	\|h(it)\|_{p_0, \tau_{N'}} 
	& \leq \sum_m |G_m(it)| \|x_m -y_m\|_{p_0, \tau_{N'}} \\
	& \leq \sum_m |G_m(it)| \|x_m -y_m\|_{p_1, \tau_{N'}} \\
	& \leq \sum_{m}\frac{\epsilon}{2^m} = O(\epsilon).
	\end{split}
	\end{equation*}
	If $0<p_0<1$, then we apply the inequality \eqref{eq:minkowski inequality p<1} instead of \eqref{eq:minkowski inequality p>1} and obtain also 
	$$\|h(it)\|^{p_0}_{p_0, \tau_{N'}}
	\leq \sum_{m}\frac{\epsilon^{p_0}}{2^{m p_0}}
	= O(\epsilon^{p_0}).$$
	So in any case, we always have $\|h(it)\|_{p_0,  \tau_{N'}}\leq O(\epsilon)$.
	Similarly we have 
	$$\|h(1+it)\|_{p_1, \tau_{N'}}\leq O(\epsilon).$$
	Thus using Lemma \ref{lem:three-line-Haagerup}, we obtain
	$$\|h(\theta)\|_{p_\theta, \tau_{N'}} \leq O(\epsilon),$$
	and moreover
	\begin{equation*}
	\left\{
	\begin{aligned}
	& \|g(it)\|_{p_0, \tau_{N'}} \le 1+ O(\epsilon) , \\
	& \|g(1+it)\|_{p_1, \tau_{N'}} \le 1+ O(\epsilon) .
	\end{aligned}
	\right.
	\end{equation*}
	Set $\tilde{G}(z) := g(z) + h(\theta)$. Then it is easy to see that $\tilde{G}\in AF(\mathcal{R}_{N'})$ and $\tilde{G}(\theta) = G(\theta)$. Moreover,
	\begin{equation*}
	\left\{
	\begin{aligned}
	& \|\tilde{G}(it)\|_{p_0, \tau_{N'}} \le 1+ O(\epsilon) , \\
	& \|\tilde{G}(1+it)\|_{p_1, \tau_{N'}} \le 1+ O(\epsilon) .
	\end{aligned}
	\right.
	\end{equation*}
	Hence from Lemma \ref{lem:three-line-Haagerup} we have $\|\tilde{G}(\theta)\|_{p_\theta, \tau_{N'}} \leq 1+ O(\epsilon)$. Thus
	$$\|G(\theta)\|_{p_\theta, \hat{\phi}} 
	=\|G(\theta)\|_{p_\theta, \tau_{N'}}  =\|\tilde{G}(\theta)\|_{p_\theta, \tau_{N'}} \leq  1+ O(\epsilon).$$
	By letting $\epsilon \to 0^+$ we deduce our result.
\end{proof}

\subsection{Applications to some matrix inequalities}

In this subsection, let $\mathcal{M}$ denote the set of all $d \times d$ complex matrices. The following proposition extends Sutter et al.'s result \cite[Theorem 3.1]{SBT2017} (see also \cite[Lemma 3.3]{JRSWW2018}) to the quasi-Banach spaces case, which generalizes Hirschman's strengthening of the Hadamard three lines theorem \cite{Hir1952}.
\begin{proposition}\label{prop:gen-three-line}
	Let $G\in AF(\mathcal{M})$.  Assume that  $0 < p_0< p_1 \leq  \infty, $ and for $\theta \in [0, 1]$ define $p_\theta$ as in the equation \eqref{eq:conjugate-index}.
	Then we have
	\begin{equation}\label{eq:hirschman's inequality}
	\log \|G(\theta) \|_{p_\theta}  \leq \int_{-\infty}^{\infty} dt \left(  \beta_{1-\theta} (t) \log \|G(it)\|_{p_0}^{1-\theta} + \beta_\theta (t) \log \|G(1+it)\|_{p_1}^\theta \right),
	\end{equation}
	where $\beta_\theta (t) = \frac{\sin (\pi \theta)} {2 \theta (\cosh (\pi t) + \cos (\pi \theta))}.$
\end{proposition}

\begin{proof}
	The proof is similar to that  of Proposition \ref{prop:three-line}. Choose an integer $n \in \mathbb{N}$ such that $n p_0\ge 1, np_1\ge 1$, then we have
	\begin{equation*}
	\begin{split}
	& \log \|G(\theta)\|_{p_\theta} \leq \sum_{k=1}^n  \log \|G_k (\theta)\|_{n p_\theta} \\
	& \leq \sum_{k=1}^n \int_{-\infty}^{\infty} dt \left(  \beta_{1-\theta} (t) \log \|G_k(it)\|_{n p_0}^{1-\theta} + \beta_\theta (t) \log \|G_k (1+it)\|_{n p_1}^\theta \right) \\
	& = \int_{-\infty}^{\infty} dt \left(  \beta_{1-\theta} (t) \log \prod_{k=1}^n \|G_k(it)\|_{n p_0}^{1-\theta} + \beta_\theta (t) \log \prod_{k=1}^n \|G_k(1+it)\|_{n p_1}^\theta \right)\\
	& \leq   \int_{-\infty}^{\infty} dt   \beta_{1-\theta} (t) \log \left( \|G(it)\|_{p_0}^{1/n}  +  \epsilon \right)^{n(1-\theta)}\\
	&+ \int_{-\infty}^{\infty} dt  \beta_\theta (t) \log \left( \|G(1+it)\|_{p_1}^{1/n} + \epsilon \right)^{n\theta},
	\end{split}
	\end{equation*}
	where $G_k\in AF(\mathcal{M}), k=1,\cdots, n$, are functions that are deduced by Corollary \ref{cor:cor of xu's result}. Note that for the first inequality we have used H\"{o}lder's inequality, and for the second one, we have used Theorem 3.1 in \cite{SBT2017}, which states that the inequality \eqref{eq:hirschman's inequality} is true for the Banach spaces case, i.e., $1\leq p_0,p_1 \leq \infty.$ We complete our proof by letting $\epsilon \to 0^+.$
\end{proof}

\begin{corollary}\label{cor:log}
	Let $0 < p \leq \infty, r \in ( 0, 1], n \in \mathbb{N},$ then for a family of positive semidefinite matrices $\{A_k\}_{k=1}^n \subset \mathcal{M}$ we have
	\begin{equation}\label{eq:log}
	\log \left\| \left| \prod_{k=1}^n A_k^r \right|^{\frac{1}{r}} \right\|_p \leq \int_{-\infty}^{\infty} dt \beta_r(t) \log \left\| \prod_{k=1}^n A_k^{1+it} \right\|_p.
	\end{equation}
\end{corollary}

\begin{proof}
	We use the same idea of \cite[Theorem 3.2]{SBT2017}.  Define the function
	$$G(z) : = \prod_{k=1}^n A_k^z.$$
	It satisfies the conditions of Proposition \ref{prop:gen-three-line}. Furthermore we pick $\theta =r, p_0=\infty$ and $p_1=p$. Then $p_\theta = p/r$ and we have
	\begin{equation*}
	\left\{
	\begin{aligned}
	& \log \|G(it)\|_{p_0}^{ 1-\theta} = (1-r)\log \left\| \prod_{k=1}^n A_k^{it} \right\|_\infty =0,\\
	& \log \|G(1+it)\|_{p_1}^\theta = r \log \left\| \prod_{k=1}^n A_k^{1+it} \right\|_p,\\
	& \text{and} \; \log \|G(\theta)\|_{p_\theta} = r \log \left\|  \left| \prod_{k=1}^n A_k^r \right|^{\frac{1}{r}} \right\|_p.
	\end{aligned}
	\right.
	\end{equation*}
	Plugging the above equations into Proposition \ref{prop:gen-three-line} deduces the corollary.
\end{proof}

\begin{remark}
	In \cite{SBT2017}, Hiai et al. proved the equation \eqref{eq:log} for any unitarily invariant norm $\| \cdot \|$ on $\mathcal{M}$ in an alternative way, which includes the conclusion of Corollary \ref{cor:log}.
\end{remark}

%%%%%%%%%%%%%%%%%%%%%%%%%%%%%%%

\section{Interpolation of noncommutative $L_p$-spaces.}

Let $\mathcal{M}$ be a ($\sigma$-finite) von Neumann algebra with a n.s.f. state $\phi.$ In this section, we will use the notation $L_p(\mathcal{M})$ for $L_p(\mathcal{M}, \phi)$ if there is no ambiguity. By embedding $\mathcal{M}_a$ into $L_p(\mathcal{M})$ via the map $x \mapsto h_\phi^{1/2p} x h_\phi^{1/2p},$ we can define our interpolation space as follows.
For $x \in \mathcal{M}_a$ Set
$$\mathcal{F} (L_{p_0}(\mathcal{M}), L_{p_1}(\mathcal{M})) := \left\{f: f(z) = \sum_{k=1}^m f_k(z) x_k, x_k \in \mathcal{M}_a, f_k \in A(S),m\in\mathbb{N} \right\},$$
equipped with the norm
$$\|f\|_{\mathcal{F} (L_{p_0}(\mathcal{M}), L_{p_1}(\mathcal{M}))} = \sup_{t \in \mathbb{R}} \left\{ \|f(it)\|_{p_0, \phi}, \|f(1+it)\|_{p_1, \phi} \right\},$$
For $0 < \theta <1$ we define the complex interpolation (quasi) norm $\|\cdot\|_\theta$ on $\mathcal{M}_a$ as follows: for $x \in \mathcal{M}_a$,
\begin{equation*}
\|x\|_{\theta}:= \inf \left\{ \|f\|_{\mathcal{F} (L_{p_0}(\mathcal{M}), L_{p_1}(\mathcal{M}))}: f(\theta) =x, f\in \mathcal{F} (L_{p_0}(\mathcal{M}), L_{p_1}(\mathcal{M})) \right\}.
\end{equation*}
Then the complex interpolation space $[L_{p_0}(\mathcal{M}), L_{p_1}(\mathcal{M})]_\theta$ is defined as the completion of $(\mathcal{M}_a, \|\cdot\|_\theta)$. The following theorem extends Xu' s result \cite {X1990} to (non-tracial) state case. 

\begin{theorem}\label{thm:interpolation2}
	Let $\mathcal{M}$ be a ($\sigma$-finite) von Neumann algebra with a n.f. state $\phi.$ For $0< p_0 <p_1 \leq \infty$, we have
	\begin{equation}
	[ L_{p_0} (\mathcal{M}) , L_{p_1}(\mathcal{M})]_{\theta} = L_{p_\theta} (\mathcal{M}),
	\end{equation}
	where $\frac{1}{p_\theta} = \frac{1-\theta}{p_0}+ \frac{\theta}{p_1}.$
\end{theorem}

\begin{proof}
	The case of $1\leq p_0 < p_1 \leq \infty$ is basically Kosaki's interpolation theorem  \cite{K1984}. 
	%Because for $1\leq p \leq \infty$ and $1/p + 1/p' =1$, the map $i_p:x\mapsto h_\phi^{1/2p'} x h_\phi^{1/2p'}$ is an isometric isomorphism from $L_p(\mathcal{M}, \phi)$ onto $C_p (\mathcal{M}, \phi)$. Due to standard techniques of the interpolation theory (see \cite{BL1976}) we can deduce our result. 
	
	The proof for the case $0< p_0 <1$ is an adaption of Xu' result \cite{X1990}. On one hand, we prove that
	\begin{equation}\label{eq:inclusion1}
	L_{p_\theta} (\mathcal{M}) \subseteq [ L_{p_0} (\mathcal{M}) , L_{p_1}(\mathcal{M})]_{\theta}.
	\end{equation}
	For this it is sufficient to show
	$\|x\|_\theta \leq \|x\|_{p_\theta, \phi}$ for any $x \in \mathcal{M}_a.$
	
	Firstly we consider the case $\frac{1}{2} \leq p_0 <1.$ To this end, using the polar decomposition, $x \in \mathcal{M}_a$ can be written as
	$$x = u |x| = u |x|^{\frac{1}{2}} |x|^{\frac{1}{2}}=x_1 x_2,$$
	where $x_1 = u |x|^{\frac{1}{2}}$ and $x_2 = |x|^{\frac{1}{2}}.$ Since $\mathcal{M}_a$ is $w^*$-dense in $\mathcal{M},$ for any $\epsilon >0$ and $k=0,1$, there exists $y_k \in \mathcal{M}_a$ such that 
	$$\|x_k - y_k\| \leq \epsilon.$$
	Thus 
	\begin{equation*}
	\left\|h_\phi^{\frac{1}{2p_\theta}} (x_1 -y_1) \right\|_{L_{2p_\theta}(\mathcal{M})}
	\leq \|x_1- y_1\| \left\|h_\phi^{\frac{1}{2p_\theta}}\right\|_{L_{2p_\theta}(\mathcal{M})}
	= \|x_1- y_1\| \leq \epsilon,
	\end{equation*}
	and similarly
	$$\left\|(x_2 -y_2) h_\phi^{\frac{1}{2p_\theta}} \right\|_{L_{2p_\theta}(\mathcal{M})} \leq \|x_2- y_2\| \leq \epsilon.$$
	On the other hand, by \eqref{eq:p norm and 2p norm}, we have
	
	\begin{equation*}
	\begin{split}
	%\left\| x_1 \right\|_{\mathcal{L}_{2p_{\theta}}(\mathcal{M})} =
	\left\| h_\phi^{\frac{1}{2p_\theta}} x_1 \right\|_{L_{2p_{\theta}}(\mathcal{M})} 
	=  \left\|h_\phi^{\frac{1}{2p_\theta}} u |x| u^* h_\phi^{\frac{1}{2p_\theta}} \right\|_{L_{p_{\theta}}(\mathcal{M})}^{\frac{1}{2}} 
	= \|x\|_{p_\theta, \phi}^{\frac{1}{2}},
	\end{split}
	\end{equation*}
	and
	\begin{equation*}
	\begin{split}
	%\left\| x_2 \right\|_{\mathcal{L}_{2p_{\theta}}(\mathcal{M})} =
	\left\| x_2 h_\phi^{\frac{1}{2p_\theta}} \right\|_{L_{2p_{\theta}}(\mathcal{M})} 
	=  \left\|  h_\phi^{\frac{1}{2p_\theta}} |x| h_\phi^{\frac{1}{2p_\theta}}\right\|_{L_{p_{\theta}}(\mathcal{M})}^{\frac{1}{2}} 
	= \|x\|_{p_\theta, \phi}^{\frac{1}{2}}.
	\end{split}
	\end{equation*}
	Therefore 
	$$\left\| h_\phi^{\frac{1}{2p_\theta}} y_1\right\|_{L_{2p_\theta}(\mathcal{M})} \leq \|x\|_{p_\theta, \phi}^{\frac{1}{2}} + \epsilon,$$
	and
	$$\left\| y_2 h_\phi^{\frac{1}{2p_\theta}}\right\|_{L_{2p_\theta}(\mathcal{M})} \leq \|x\|_{p_\theta, \phi}^{\frac{1}{2}} + \epsilon.$$

	%By \eqref{eq:p norm and 2p norm}, we have
	
	%\begin{equation*}
	%\begin{split}
	%\left\| x_1 \right\|_{\mathcal{L}_{2p_{\theta}}(\mathcal{M})} =
	%\left\| h_\phi^{\frac{1}{2p_\theta}} x_1 \right\|_{L_{2p_{\theta}}} 
	%=  \left\|h_\phi^{\frac{1}{2p_\theta}} u |x| u^* h_\phi^{\frac{1}{2p_\theta}} \right\|_{L_{p_{\theta}}}^{\frac{1}{2}} 
	%= \|x\|_{p_\theta, \phi}^{\frac{1}{2}},
	%\end{split}
	%\end{equation*}
	%and
	%\begin{equation*}
	%\begin{split}
	%\left\| x_2 \right\|_{\mathcal{L}_{2p_{\theta}}(\mathcal{M})} =
	%\left\| x_2 h_\phi^{\frac{1}{2p_\theta}} \right\|_{L_{2p_{\theta}}} 
	%=  \left\|  h_\phi^{\frac{1}{2p_\theta}} |x| h_\phi^{\frac{1}{2p_\theta}}\right\|_{L_{p_{\theta}}}^{\frac{1}{2}} 
	%= \|x\|_{p_\theta, \phi}^{\frac{1}{2}}.
	%\end{split}
	%\end{equation*}
	%So $h_\phi^{\frac{1}{2p_\theta}} x_1,x_2 h_\phi^{\frac{1}{2p_\theta}} \in L_{p_\theta} (\mathcal{M})$. Then by Lemma \ref{lem:dense}, for any $\epsilon>0$, there exist $y_k \in \mathcal{M}_a, k=1,2$ such that 
	%$$\left\|h_\phi^{\frac{1}{2p_\theta}} (x_1 -y_1) \right\|_{L_{2p_\theta}} \leq \epsilon,$$ 
	%and
	%$$\left\|(x_2 -y_2) h_\phi^{\frac{1}{2p_\theta}} \right\|_{L_{2p_\theta}} \leq \epsilon.$$
	%Therefore 
	%$$\left\| h_\phi^{\frac{1}{2p_\theta}} y_1\right\|_{L_{2p_\theta}} \leq \|x\|_{p_\theta, \phi}^{\frac{1}{2}} + \epsilon,$$
	%and
	%$$\left\| y_2 h_\phi^{\frac{1}{2p_\theta}}\right\|_{L_{2p_\theta}} \leq \|x\|_{p_\theta, \phi}^{\frac{1}{2}} + \epsilon.$$
	
	Let $z_1 = \sigma^\phi_{-\frac{i}{4p_\theta}} (y_1)$ and $z_2 = \sigma^\phi_{\frac{i}{4p_\theta}} (y_2),$ then due to the equality \eqref{eq:proof of lem of Junge-Xu} we have $h_\phi^{\frac{1}{4p_\theta}} z_1 h_\phi^{\frac{1}{4p_\theta}} = h_\phi^{\frac{1}{2p_\theta}} y_1$ and 
	$h_\phi^{\frac{1}{4p_\theta}} z_2 h_\phi^{\frac{1}{4p_\theta}} = y_2 h_\phi^{\frac{1}{2p_\theta}}.$ Moreover, $z_k \in \mathcal{M}_a$ and 
	$$\|z_k\|_{2p_\theta, \phi}  \leq \|x\|_{p_\theta, \phi}^{\frac{1}{2}} + \epsilon, k =1, 2.$$

	By the first part, $[L_{2p_0}(\mathcal{M}) , L_{2p_1}(\mathcal{M})]_\theta = L_{2p_\theta}(\mathcal{M})$. Then for any $\epsilon>0$ and $k=1, 2$, there exists $f_k \in \mathcal{F}(L_{2p_0}(\mathcal{M}) , L_{2p_1}(\mathcal{M})),$ such that $f_k(\theta) = z_k,$ and 
	%\begin{equation*}
	%\left\{
	%\begin{aligned}
	%& \|f_1\|_{ \mathcal{F}(L_{2p_0}(\mathcal{M}) , L_{2p_1}(\mathcal{M}))}\leq \| x_1  \|_{L_{2p_\theta}(\mathcal{M})}+\epsilon =  \|x\|_{L_{p_{\theta}}(\mathcal{M})}^{\frac{1}{2}}+\epsilon,\\
	%& \|f_2\|_{ \mathcal{F}(L_{2p_0}(\mathcal{M}) , L_{2p_1}(\mathcal{M}))}\leq \| x_2  \|_{L_{2p_\theta}(\mathcal{M})} +\epsilon =  \|x\|_{L_{p_{\theta}}(\mathcal{M})}^{\frac{1}{2}}+\epsilon.
	%\end{aligned}
	%\right.
	%\end{equation*}
	\begin{equation*}
	\begin{split}
	\|f_k\|_{ \mathcal{F}(L_{2p_0}(\mathcal{M}) , L_{2p_1}(\mathcal{M}))} & \leq \| z_k  \|_{2p_\theta, \phi}+\epsilon \\
	& \leq  \|x\|_{p_\theta, \phi}^{\frac{1}{2}}+\epsilon.
	\end{split}
	\end{equation*}
	Now consider
	$$f(z):=  \sigma^\phi_{\frac{i}{4p_z}} (f_1(z))  \sigma^\phi_{-\frac{i}{4p_z}} ( f_2 (z)) - y_1y_2 +x.$$
	It is easy to see that $f$ is analytic in $S$, continuous on $\bar{S}$, and $f(\theta) =x$. 
	Moreover, from \eqref{eq:minkowski inequality p<1} and H\"{o}lder's inequality \eqref{eq:holder} it follows that
	\begin{equation*}
	\begin{split}
	\|f(it)\|_{p_0, \phi}^{p_0} & \leq \left\|h_\phi^{\frac{1}{2p_0}}  \sigma^\phi_{\frac{i}{4p_0}} (f_1(it))  \sigma^\phi_{-\frac{i}{4p_0}} ( f_2 (it)) h_\phi^{\frac{1}{2p_0}} \right\|_{L_{p_0}(\mathcal{M})}^{p_0} + \|x -y_1y_2 \|_{p_0, \phi}^{p_0}\\
	& \leq \left\|h_\phi^{\frac{1}{4p_0}}  f_1(it) h_\phi^{\frac{1}{2p_0}}  f_2 (it) h_\phi^{\frac{1}{4p_0}} \right\|_{L_{p_0}(\mathcal{M})}^{p_0} + \|x -y_1y_2 \|_{p_0, \phi}^{p_0}\\
	& \leq \|f_1(it)\|_{2p_0, \phi}^{p_0}  \|f_2(it)\|_{2p_0, \phi}^{p_0} + \|x -y_1y_2 \|_{p_0, \phi}^{p_0}\\
	& \leq \|x\|_{p_\theta, \phi}^{p_0} + \|x -y_1y_2 \|_{p_0, \phi}^{p_0} + O(\epsilon)\\
	& \leq \|x\|_{p_\theta, \phi}^{p_0}  + \|x -y_1y_2 \|^{p_0} + O(\epsilon)\\
	& \leq \|x\|_{p_\theta, \phi}^{p_0} + O(\epsilon).
	\end{split}
	\end{equation*} 
	
	Similarly we have (replace \eqref{eq:minkowski inequality p<1} with \eqref{eq:minkowski inequality p>1} if necessary)
	$$ \|f(1+it)\|_{p_1, \phi} \leq \|x\|_{p_\theta, \phi}+O(\epsilon).$$ 
	Therefore
	$$\|x\|_\theta 
	\leq \|f \|_{\mathcal{F}(L_{p_0}(\mathcal{M}) , L_{p_1}(\mathcal{M}))} 
	\leq \|x\|_{p_\theta, \phi}+O(\epsilon).$$
	Thus we can conclude the inclusion \eqref{eq:inclusion1} for $\frac{1}{2} \leq p_0 < 1$ by letting $\epsilon\to 0^+$.
	
	By repeating the above arguments we can obtain the inclusion \eqref{eq:inclusion1} for $\frac{1}{4} \leq p_0< \frac{1}{2}.$ Thus by iterating this procedure, we get the inclusion \eqref{eq:inclusion1} for all $0 < p_0<1.$
	
	On the other hand, we show that
	\begin{equation}\label{eq:inclusion2}
	[ L_{p_0} (\mathcal{M}) , L_{p_1}(\mathcal{M})]_{\theta} \subseteq L_{p_\theta} (\mathcal{M}).
	\end{equation}
	To this end, it suffices to show that $\|x\|_{p_\theta, \phi} \leq \|x\|_\theta$ for any $x \in \mathcal{M}_a.$ Thus for any $f \in \mathcal{F}(L_{p_0}(\mathcal{M}), L_{p_1}(\mathcal{M}))$ such that
	$f(\theta) =x,$ we will show
	$$\|x\|_{p_\theta, \phi} \leq \|f\|_{\mathcal{F}(L_{p_0}(\mathcal{M}), L_{p_1}(\mathcal{M}))}.$$
	Without loss of generality, suppose $\|f\|_{\mathcal{F}( L_{p_0}(\mathcal{M}), L_{p_1}(\mathcal{M}))} \leq 1.$ Therefore using Proposition \ref{prop:three-line-Haagerup}, we have
	\begin{equation*}
	\|x\|_{p_\theta, \phi} 
	= \left\|  f(\theta) \right\|_{p_\theta, \phi}\\
	%&= \|F(\theta)\|_{L_{p_\theta}(\mathcal{M})}\\
	\leq \|f(it)\|_{p_0, \phi}^{1-\theta} \|f(1+it)\|_{p_1, \phi}^\theta \leq 1,
	\end{equation*}
	which completes our proof.
\end{proof}

\begin{remark}
	\begin{enumerate}[(i)] 
		\item Instead of the symmetric embedding for $\mathcal{M}$ into $L_p(\mathcal{M}, \phi)$, our interpolation also holds for any embedding $x \mapsto  h_\phi^{\frac{1-\eta}{2p}}  x h_\phi^{\frac{\eta}{2p}}$
		as in Lemma  \ref{lem:dense}, $0\leq \eta \leq1.$ 
		\item The n.f. state $\phi$ can be replaced by a weight. 
	\end{enumerate}
\end{remark}

We have the following direct corollaries.

\begin{corollary}\label{cor:fixed}
	Let $0 < p_0 < p_1 \leq \infty, $ and for $\theta \in [0, 1]$ define $p_\theta$ as in the equation \eqref{eq:conjugate-index}. For any $x \in \mathcal{M}$ we have
	\begin{equation*}
	\left\|h^{\frac{1}{2p_\theta}}_\phi xh^{\frac{1}{2p_\theta}}_\phi \right\|_{L_{p_\theta}(\mathcal{M})} \leq \left\| h_\phi^{\frac{1}{2p_0}} x h_\phi^{\frac{1}{2p_0}} \right\|_{L_{p_0}(\mathcal{M})}^{1-\theta} \left\|h_\phi^{\frac{1}{2p_1}} x h_\phi^{\frac{1}{2p_1}}\right\|_{L_{p_1}(\mathcal{M})}^{\theta}.
	\end{equation*}
\end{corollary}

\begin{proof}
	By taking the map $G(z) =  x $ in Proposition \ref{prop:three-line-Haagerup} we can deduce our result.
\end{proof}

Let $\mathcal{M}, \mathcal{N}$ be two ($\sigma$-finite) von Neumann algebras, and $\phi$ (resp. $\psi$) be a n.f. state on $\mathcal{M}$ (resp. $\mathcal{N}$). Let $T: \mathcal{M} \to \mathcal{N}$ be a linear map. Then we define an operator norm $\|T\|_{(p, \phi) \to (q, \psi)}$ by
$$\|T\|_{(p, \phi) \to (q, \psi)} = \sup_{X \neq 0} \frac{\|T(X)\|_{q, \psi}}{\|X\|_{p, \phi}}.$$	

\begin{corollary}\label{cor:RT}
	Let $\mathcal{M}, \mathcal{N}$ be two ($\sigma$-finite) von Neumann algebras, and $\phi$ (resp. $\psi$) be a n.s.f. state on $\mathcal{M}$ (resp. $\mathcal{N}$). Let $T: \mathcal{M} \to \mathcal{N}$ be a linear map. 	
	Assume that $0 < p_0 < p_1 \leq \infty, 0< q_0 < q_1 \leq \infty$ and $p_\theta, q_\theta$ satisfy the equation \eqref{eq:conjugate-index}. Then we have
	\begin{equation*}
	\begin{split}
	\|T\|_{(p_\theta, \phi) \to (q_\theta, \psi)}  \leq \|T\|_{(p_0, \phi) \to (q_0, \psi)}^{1-\theta} \cdot \|T\|_{(p_1, \phi) \to (q_1, \psi)} ^\theta.
	\end{split}
	\end{equation*}
\end{corollary}

\begin{proof}
	It is a direct corollary of Theorem \ref{thm:interpolation2} using the standard approach in the interpolation theory \cite{BL1976}. 
\end{proof}

%%%%%%%%%%%%%%%%%%%%%%%%%%%%%%%%

\section{Applications to the Sandwiched R\'enyi relative entropy}
The history of the generalization of $p$-R\'{e}nyi divergence can be traced back to Petz's work \cite{Petz1986, Petz1993}, where he extended the notion of the $p$-R\'{e}nyi divergence to the general von Neumann algebra context.
Regarding to the sandwiched $p$-R\'{e}nyi divergence, there are mainly two approaches. 

One is introduced by Jencov\'{a} \cite{Jencova2017, Jencova2018} as follows
\begin{equation}\label{eq:def of Jencova divergence}
\tilde{D}_p^J (\psi \| \phi): = \frac{1}{p-1} \log \left\| h_\psi \right\|^p_{C_p(\mathcal{M}, \phi)}, \;  p \in (1, \infty),
\end{equation}
where $C_p(\mathcal{M}, \phi)$ is the Kosaki's noncommutative $L_p$-space associated with $(\mathcal{M}, \phi).$ Recall that
\begin{equation*}
C_p(\mathcal{M}, \phi) := \left[h_\phi^{\frac{1}{2}} \mathcal{M} h_\phi^{\frac{1}{2}}, L_1(\mathcal{M}, \phi) \right]_{1/p},
\end{equation*}
where $1\leq p \leq \infty$. Note that $C_\infty (\mathcal{M}, \phi) = h_\phi^{\frac{1}{2}} \mathcal{M} h_\phi^{\frac{1}{2}},$ and for any $x \in \mathcal{M},$ $\left\|h_\phi^{\frac{1}{2}} x h_\phi^{\frac{1}{2}}\right\|_{C_\infty (\mathcal{M}, \phi)} = \|x\|.$ Since the definition \eqref{eq:def of Jencova divergence} relies on Kosaki's construction of noncommutative $L_p$-space, the index $p$ can not go beyond the interval $(1, \infty).$

Another definition is derived by Berta, Scholz and Tomamichel \cite{BST2016} via a so-called \emph{noncommutative vector valued $L_p$-space}, denoted by $L_p^{BST}(\mathcal{M}, \phi)$. The definition of spaces $L_p^{BST}(\mathcal{M}, \phi)$ relies on the \emph{spatial derivative} $\triangle (\xi/ \phi),$ where $\xi$ is a vector in $L_2 (\mathcal{M}, \phi).$ The definition of the space $L_p^{BST}(\mathcal{M}, \phi)$ is subtle. We omit the details and refer to \cite{BST2016}. They defined the divergence as follows
\begin{equation}
\tilde{D}_p^{BST} (\psi \| \phi): = \frac{2p}{p-1} \log \left\| \xi_\psi \right\|_{L_{2p}^{BST}(\mathcal{M}, \phi)}, \;  p \in [1/2, 1) \cup (1, \infty),
\end{equation}
where $\xi_\psi$ is a vector representation of $\psi$ for a $*$-representation $\pi: \mathcal{M} \to B(H).$

%\begin{remark}
%	In principle, the noncommutative $L_p$-spaces involving in the definitions of $\tilde{D}_p^J (\psi \| \phi)$ and $\tilde{D}_p^{BST} (\psi \| \phi)$ are Banach spaces ($1\leq p < \infty$).
%\end{remark}

Now let $\mathcal{M}$ be a $\sigma$-finite von Neumann algebra acting on $H$, and $\phi, \psi$ be two normal faithful states on $\mathcal{M}.$ For $p \in (0, 1) \cup (1, \infty)$ we define
\begin{equation}\label{eq:relative-entropy}
\begin{split}
\tilde{D}_p (\psi \| \phi) : 
&  = \frac{1}{p-1} \log  \left\| h_\phi^{\frac{1-p}{2p}} h_\psi h_\phi^{\frac{1-p}{2p}} \right\|^p_{L_p(\mathcal{M}, \phi)}\\
& =  \frac{1}{p-1} \log  \left\| h_\phi^{-\frac{1}{2}} h_\psi h_\phi^{-\frac{1}{2}} \right\|^p_{p, \phi},
\end{split}
\end{equation}
where $h_\psi$ and $h_\phi$ are Randon-Nikodym derivatives of dual weights $\hat{\psi}$ and $\hat{\phi}$ with respect to the canonical trace $\tau$ on $\mathcal{M}\rtimes_{\sigma^\phi}\mathbb{R}$, respectively.

% Relations with other definitions	

Using Lemma \ref{lem:isometry}, we can show that the definition of \eqref{eq:relative-entropy} extends the notion of the sandwiched $p$-R\'enyi divergence to the $\sigma$-finite von Neumann algebra case. We recall that for two positive operators $\sigma, \rho \in B(H)$, the sandwiched $p$-R\'enyi divergence is defined as \cite{MDSFT2013}
\begin{equation}
\tilde{D}_p (\rho \| \sigma) = \frac{1}{p-1} \log  \left\| \sigma^{\frac{1-p}{2p}} \rho \sigma^{\frac{1-p}{2p}} \right\|^p_p, \; p \in (0, 1) \cup (1, \infty),
\end{equation}
where $\|\cdot\|_p$ is the Schatten $p$-norm on $B(H).$

\begin{proposition}\label{prop:covering Beigi's result}
	Let $\mathcal{M}$ be a finite von Neumann algebra with a n.f. tracial state $\tau$, and $\psi, \phi$ be two n.f. states on $\mathcal{M}$ defined by
	$$\psi(x)=\tau(\rho x),~~\phi(x)=\tau(\sigma x),~~x\in\mathcal{M},$$
	where $\rho$ and $\sigma$ are two positive operators with unit trace. Then
	\begin{equation}
	\tilde{D}_p (\psi \| \phi) = \frac{1}{p-1} \log  \tau \left[ \left( \sigma^{\frac{1-p}{2p}} \rho \sigma^{\frac{1-p}{2p}} \right)^p \right].
	\end{equation}
\end{proposition}

\begin{proof}
	By Lemma \ref{lem:isometry}, there exists an isometry $\theta':\Lambda_p(\mathcal{M},\tau)\to L_p(\mathcal{M},\phi)$ such that 
	$$\theta'(\sigma)=h_{\phi},~~\theta'(\rho)=h_{\psi}.$$ 
	Then we have
	\begin{equation*}
	\begin{split}
	\left\| h_\phi^{\frac{1-p}{2p}} h_\psi h_\phi^{\frac{1-p}{2p}} \right\|^p_{L_p(\mathcal{M}, \phi)} 
	& = \left\|\theta' \left( \sigma^{\frac{1-p}{2p}}  \rho \sigma^{\frac{1-p}{2p}} \right) \right\|^p_{L_p(\mathcal{M}, \phi)} \\
	& = \left\| \sigma^{\frac{1-p}{2p}}  \rho \sigma^{\frac{1-p}{2p}} \right\|^p_{\Lambda_p(\mathcal{M}, \tau)}\\
	& = \tau \left[ \left( \sigma^{\frac{1-p}{2p}} \rho \sigma^{\frac{1-p}{2p}} \right)^p \right],
	\end{split}
	\end{equation*}
	which completes our proof.
\end{proof}

\begin{proposition}
	Our definition $\tilde{D}_p (\psi \| \phi)$ fits well with the ones defined by Jen\v{c}ov\'a and Berta et al.. More precisely, we have
	\begin{enumerate}[(i)] 
		\item $\tilde{D}_p (\psi \| \phi) = \tilde{D}_p^J (\psi \| \phi)$ for $1< p < \infty;$
		
		\item $\tilde{D}_p (\psi \| \phi) = \tilde{D}_p^{BST} (\psi \| \phi)$ for $1/2 \leq p <1.$
	\end{enumerate}
\end{proposition}

\begin{proof}
	We note that (ii) has been proved by Jen\v{c}ov\'{a} in \cite[Theorem 5]{Jencova2017}.
	
	For (i), we will use the following fact proved by Kosaki \cite{K1984}: For $1\leq p \leq \infty$ and $1/p + 1/p' =1$, the map $i_p:x\mapsto h_\phi^{1/2p'} x h_\phi^{1/2p'}$ is an isometric isomorphism from $L_p(\mathcal{M}, \phi)$ onto $C_p (\mathcal{M}, \phi)$. From this it follows that
	\begin{equation}\label{eq:Jencova1}
	\left\| h_\phi^{\frac{1-p}{2p}} h_\psi h_\phi^{\frac{1-p}{2p}} \right\|_{L_p(\mathcal{M}, \phi)} 
	= \left\| i_p \left( h_\phi^{\frac{1-p}{2p}} h_\psi h_\phi^{\frac{1-p}{2p}}  \right) \right\|_{C_p(\mathcal{M},  \phi)} 
	= \left\| h_\psi \right\|_{C_p(\mathcal{M}, \phi)}.
	\end{equation}
	Hence for $1< p < \infty,$ the equation \eqref{eq:relative-entropy} coincides with Jen\v{c}ov\'a's definition, i.e.,
	\begin{equation}\label{eq:Jencova2}
	\tilde{D}_p (\psi \| \phi) = \frac{1}{p-1} \log \|h_\psi\|_{C_p(\mathcal{M}, \phi)}^p = \tilde{D}_p^J (\psi \| \phi).
	\end{equation}
	
\end{proof}

\begin{proposition}(Monotonicity in $p$). For all n.f. states $\psi, \phi$ on $\mathcal{M},$ the function $p \mapsto \tilde{D}_p ( \psi || \phi)$ is increasing for $p \in (0, 1) \cup (1, \infty).$
\end{proposition}

\begin{proof}
	For $p \in (1, \infty),$ by a direct modification of the proof in \cite[Theorem 7.1]{Beigi2013} we can deduce our result. It was also proved by Jen\v{c}ov\'a \cite{Jencova2018}.
	
	Since the logarithm function is non-decreasing, it remains to show that when $0<p_0<p_1<1$ or $0<p_0<1<p_1$, we have
	\begin{equation}\label{eq:monotone in p}
	\left\| h_\phi^{\frac{1-p_0}{2p_0}} h_\psi h_\phi^{\frac{1-p_0}{2p_0}}\right\|_{L_{p_0}(\mathcal{M}, \phi)}^{p_0'} 
	\leq \left\| h_\phi^{\frac{1-p_1}{2p_1}} h_\psi h_\phi^{\frac{1-p_1}{2p_1}} \right\|_{L_{p_1}(\mathcal{M}, \phi)}^{p_1'}.
	\end{equation}
	Here we use $s'$ to denote the conjugate number of $s$. Note that for $s \in (0,1), s'= \frac{s}{s-1}<0.$
	
	If $0< p_0< p_1 <1,$ then there exists $0< \theta < 1$ such that
	\begin{equation}\label{eq:conjugate-1}
	\frac{1}{p_1} = \frac{1-\theta}{p_0} + \frac{\theta}{1}.
	\end{equation}
	Using Corollary \ref{cor:fixed} (by letting $x = h_\phi^{-\frac{1}{2}} h_\psi h_\phi^{-\frac{1}{2}}$) we have
	$$\left\|  h_\phi^{\frac{1-p_1}{2p_1}} h_\psi h_\phi^{\frac{1-p_1}{2p_1}}\right\|_{L_{p_1}(\mathcal{M}, \phi)} 
	\leq \left\| h_\phi^{\frac{1-p_0}{2p_0}} h_\psi h_\phi^{\frac{1-p_0}{2p_0}} \right\|^{1-\theta}_{L_{p_0}(\mathcal{M}, \phi)}   \left\| h_\psi \right\|_{L_1(\mathcal{M}, \phi)}^{\theta}.$$
	Since $\|h_\psi\|_{L_1(\mathcal{M}, \phi)}=1$ and $p_1'<0$, we have
	$$\left\| h_\phi^{\frac{1-p_1}{2p_1}} h_\psi h_\phi^{\frac{1-p_1}{2p_1}} \right\|_{L_{p_1}(\mathcal{M}, \phi)}^{p_1'} 
	\geq \left\| h_\phi^{\frac{1-p_0}{2p_0}} h_\psi h_\phi^{\frac{1-p_0}{2p_0}} \right\|_{L_{p_0}(\mathcal{M}, \phi)}^{(1-\theta) p_1'}.$$
	By equation \eqref{eq:conjugate-1}, we have $(1-\theta) p_1' = p_0'$, which gives \eqref{eq:monotone in p}.
	
	The proof for $0<p_0<1<p_1$ is similar. In this case there exists $0< \theta < 1$ such that
	\begin{equation}\label{eq:conjugate-2}
	\frac{1}{1} = \frac{1-\theta}{p_0} + \frac{\theta}{p_1}.
	\end{equation}
	Using Corollary \ref{cor:fixed} we have
	$$1=\left\|  h_\psi  \right\|_{L_{1}(\mathcal{M}, \phi)} 
	\leq \left\|  h_\phi^{\frac{1-p_0}{2p_0}} h_\psi h_\phi^{\frac{1-p_0}{2p_0}} \right\|^{1-\theta}_{L_{p_0}(\mathcal{M}, \phi)}   \left\|h_\phi^{\frac{1-p_1}{2p_1}} h_\psi h_\phi^{\frac{1-p_1}{2p_1}} \right\|_{L_{p_1}(\mathcal{M}, \phi)}^{\theta}.$$
	Recall that $p_0'<0$, then 
	$$\left\|  h_\phi^{\frac{1-p_0}{2p_0}} h_\psi h_\phi^{\frac{1-p_0}{2p_0}} \right\|^{p_0'}_{L_{p_0}(\mathcal{M}, \phi)}   \leq \left\|h_\phi^{\frac{1-p_1}{2p_1}} h_\psi h_\phi^{\frac{1-p_1}{2p_1}} \right\|_{L_{p_1}(\mathcal{M}, \phi)}^{\frac{\theta}{\theta-1}p_0'}.$$
	From \eqref{eq:conjugate-2} we have $(\theta-1)p_1'=\theta p_0'$ and then obtain \eqref{eq:monotone in p}.
\end{proof}

% Data processing inequality (DPI)

We fix some notations that will be used in the remaining part of the paper. Let $\mathcal{M}$ be a $\sigma$-finite von Neumann algebra with a normal faithful state $\phi$. Let $\mathcal{N}$ be a von Neumann subalgebra of $\mathcal{M}$ and set $\phi_{\mathcal{N}}: = \phi|_\mathcal{N}.$  Let $\tau$ (resp. $\tau_{\mathcal{N}}$) be the canonical trace on $\mathcal{R}:=\mathcal{M}\rtimes_{\sigma^\phi}\mathbb{R}$ (resp. $\mathcal{R}_{\mathcal{N}}:=\mathcal{N}\rtimes_{\sigma^{\phi_{\mathcal{N}}}}\mathbb{R}$). Let $h_\psi$ (resp. $h_{\psi_N}$) be the Randon-Nikodym derivative of dual weight $\hat{\phi}$ (resp. $\widehat{\phi_{\mathcal{N}}}$) with respect to $\tau$ (resp. $\tau_{\mathcal{N}}$).

\begin{theorem}
	Let $\mathcal{N}$ be a von Neumann subalgebra of $\mathcal{M}$ and $\psi, \phi$ be n.f. states on $\mathcal{M}.$ Then for $p \in (1, \infty),$ we have the following data processing inequality (DPI):
	\begin{equation*}
	\tilde{D}_p  (\psi_{\mathcal{N}} \| \phi_{\mathcal{N}} ) \leq \tilde{D}_p  (\psi \| \phi ).
	\end{equation*}
\end{theorem}

\begin{proof}
	Basically it is an adaption of Beigi's proof \cite{Beigi2013}. It is sufficient to prove that
	\begin{equation}\label{eq:DPI}
	\left\| h_{\phi_{\mathcal{N}}}^{\frac{1-p}{2p}} h_{\psi_{\mathcal{N}}} h_{\phi_{\mathcal{N}}}^{\frac{1-p}{2p}} \right\|_{L_p(\mathcal{N}, \phi_{\mathcal{N}})}
	\leq  \left\| h_{\phi}^{\frac{1-p}{2p}} h_{\psi} h_{\phi}^{\frac{1-p}{2p}} \right\|_{L_p(\mathcal{M}, \phi)},
	\end{equation}
	for $1< p < \infty.$
	%Using the equation \eqref{eq:Jencova1} and Lemma \ref{lem:expectation}, it is equivalent to show that
	%$$ \left\| \epsilon_* (h_\psi) \right\|_{C_p(\mathcal{M}, \phi \circ \epsilon)} \leq  \left\| h_\psi \right\|_{C_p(\mathcal{N}, \phi)}.$$
	
	Recall that $\mathcal{M}_*=L_1(\mathcal{M}, \phi)$. Then for any $x = h_\omega \in L_1(\mathcal{M}, \phi)$ with $\omega\in \mathcal{M}_*$, define a map $T: L_1(\mathcal{M}, \phi) \to L_1(\mathcal{N}, \phi_{\mathcal{N}})$ by
	$$T (x):= h_{\omega_\mathcal{N}}.$$
	It is easy to see that $T$ is positive and $\tr$-preserving on $L_1(\mathcal{M}, \phi)$. Consider the map $S:  L_1(\mathcal{M}, \phi) \to L_1(\mathcal{N}, \phi_{\mathcal{N}})$ given by
	$$S (x) =h_{\phi_{\mathcal{N}}}^{-\frac{1}{2}} T \left( h_{\phi}^{\frac{1}{2}} x h_{\phi}^{\frac{1}{2}} \right) h_{\phi_{\mathcal{N}}}^{-\frac{1}{2}}.$$
	Clearly $S \left(h_{\phi}^{-\frac{1}{2}} h_{\psi} h_{\phi}^{-\frac{1}{2}} \right) = h_{\phi_{\mathcal{N}}}^{-\frac{1}{2}} h_{\psi_{\mathcal{N}}} h_{\phi_{\mathcal{N}}}^{-\frac{1}{2}}.$ Using Corollary \ref{cor:RT}, to show \eqref{eq:DPI} it is sufficient to prove that
	$$\|S\|_{(p, \phi) \to (p, \phi_\mathcal{N})} \leq 1,$$
	for $p =1$ and $p=\infty.$
	
	For $p= 1$, since $T$ is $\tr$-preserving, we have
	\begin{equation}\label{eq:L1}
	\| S (x) \|_{1, \phi_{\mathcal{N}}} = \left\| T \left( h_{\phi}^{\frac{1}{2}} x h_{\phi}^{\frac{1}{2}} \right) \right\|_{1, \phi_{\mathcal{N}}}
	= \left\|  h_{\phi}^{\frac{1}{2}} x h_{\phi}^{\frac{1}{2}}  \right\|_{L_1(\mathcal{M}, \phi)} = \|x\|_{1, \phi}.
	\end{equation}
	
	For $p =\infty,$ using the positivity of $T$, we have
	$$\|S\|_{(\infty, \phi) \to (\infty, \phi_\mathcal{N})}
	=\|S(1)\|_{\mathcal{N}}
	=\|h_{\phi_{\mathcal{N}}}^{-\frac{1}{2}}  h_{\phi_{\mathcal{N}}} h_{\phi_{\mathcal{N}}}^{-\frac{1}{2}}\|_{\mathcal{N}}
	=1.\qedhere$$
\end{proof}

\begin{remark}
	It is known that \cite{FL2013} the sandwiched $p$-R\'enyi divergence satisfies DPI for $p\in(\frac{1}{2},1)\cup(1,\infty)$. DPI for a more general family of $p$-R\'enyi divergences, knowns as \textit{$\alpha$-$z$ R\'enyi relative entropies}, was determined by the last author in \cite{Z2018}. See also a survey paper \cite{CFL2018} for more information of DPI.
\end{remark}

Suppose $\mathcal{N}$ is invariant under the automorphism group $\sigma_t^\phi,$ i.e., $\sigma_t^\phi (\mathcal{N}) = \mathcal{N}$ for all $t\in\mathbb{R}$. Then due to the modular theory \cite{Takesaki}, there exists a  conditional expectation $\mathcal{E}: \mathcal{M} \to \mathcal{N}$ such that $\phi \circ \mathcal{E} = \phi.$ Moreover \cite[Theorem 4.1]{HJX2010}, $\mathcal{E}$ can be naturally extended to a conditional expectation  $\hat{\mathcal{E}}: \mathcal{R} \to \mathcal{R_N}$ satisfying
\begin{equation}\label{eq:trace-preserve}
\tau_{\mathcal{N}} [\hat{\mathcal{E}}(y)] = \tau (y),
\end{equation}
for any $ y \in \mathcal{R}.$

Denote by $\iota:\mathcal{R_N} \to \mathcal{R}$ the natural embedding.
By \eqref{eq:trace-preserve}, we have
\begin{equation}\label{eq:5.14}
\tau [\iota (x) y] = \tau_{\mathcal{N}} [\hat{\mathcal{E}}(\iota (x) y)] = \tau_{\mathcal{N}} [ \hat{\mathcal{E}} (x y)] = \tau_{\mathcal{N}} [x \hat{\mathcal{E}} (y)],
\end{equation}
for any $x \in\mathcal{R_N}, y \in \mathcal{R}$.

\begin{proposition}\label{prop:expectation}
	Let $\mathcal{N}$ be a von Neumann subalgebra of $\mathcal{M}$ and $\psi, \phi$ be n.f. states on $\mathcal{M}.$ Suppose that $\sigma_t^\phi (\mathcal{N}) =\mathcal{N}$ for all $t\in\mathbb{R}$, and let $\mathcal{E}:\mathcal{M}\to\mathcal{N}$ be the $\phi$-preserving conditional expectation. Then we have
	\begin{equation*}
	\hat{\mathcal{E}} (h_{\psi}) = h_{\psi_{\mathcal{N}}}.
	\end{equation*}
\end{proposition}

\begin{proof}
	Since $\phi \circ \mathcal{E} = \phi$, we have 
	$\mathcal{E} \circ \sigma_t^{\phi} = \sigma_t^{\phi_{\mathcal{N}}} \circ \mathcal{E}.$
	Thus by a result of Haagerup, Junge and Xu \cite[Theorem 4.1]{HJX2010}, we have $\hat{\mathcal{E}}\circ\hat{\sigma}_{t}^{\phi}= \hat{\sigma}_t^{{\phi_{\mathcal{N}}}}\circ\hat{\mathcal{E}}$. It follows immediately that
	\begin{equation*}
	\hat{\sigma}_t^{\phi} |_\mathcal{N}= \hat{\sigma}_t^{{\phi_{\mathcal{N}}}}.
	\end{equation*}
	Therefore for any $ x \in \mathcal{R_N}, x \geq 0,$
	\begin{equation*}
	\begin{split}
	\tau_{\mathcal{N}} [ \hat{\mathcal{E}} (h_{\psi})  x]  
	& = \tau (h_{\psi} x) \\
	&= \psi  \left[ \int_{-\infty}^{\infty} \hat{\sigma}_t^{\phi} (x) dt \right]  = \psi  \left[ \int_{-\infty}^{\infty} \hat{\mathcal{E}} \circ \hat{\sigma}_t^{{\phi_{\mathcal{N}}}} (x) dt \right] \\
	&=  \psi_{\mathcal{N}}  \left[ \int_{-\infty}^{\infty} \hat{\sigma}_t^{{\phi_{\mathcal{N}}}} (x) dt \right] = \tau_{\mathcal{N}} [h_{\psi_{\mathcal{N}}}  x].
	\end{split}
	\end{equation*}
	This completes our proof, since $\tau_\mathcal{N}$ is faithful.
\end{proof}

\begin{proposition}\label{prop:embed}
	Let $\mathcal{N}$ be a von Neumann subalgebra of $\mathcal{M}$ and $\psi, \phi$ be n.f. states on $\mathcal{M}.$ Suppose that $\sigma_t^\phi (\mathcal{N}) = \mathcal{N}$ for all $t\in\mathbb{R}$, and let $\mathcal{E}:\mathcal{M}\to\mathcal{N}$ be the $\phi$-preserving conditional expectation. Then we have 
	\begin{equation*}
	\iota (h_{\psi_{\mathcal{N}}}) = h_{\psi_{\mathcal{N}} \circ \mathcal{E}}.
	\end{equation*}
\end{proposition}

\begin{proof}
	Since $\mathcal{E} \circ \sigma_t^{\phi} = \sigma_t^{\phi_{\mathcal{N}}} \circ \mathcal{E}$, we have by \cite[Theorem 4.1]{HJX2010} that
	\begin{equation*}
	\hat{\mathcal{E}}  \circ \hat{\sigma}_t^{\phi} = \hat{\sigma}_t^{{\phi_{\mathcal{N}}}} \circ \hat{\mathcal{E}}.
	\end{equation*}
	Therefore for any $ y \in  \mathcal{R}, y \geq 0,$
	\begin{equation*}
	\begin{split}
	\tau [h_{\psi_{\mathcal{N}} \circ \mathcal{E}} y] & = \psi_{\mathcal{N}} \circ \mathcal{E}  \left[ \int_{-\infty}^{\infty} \hat{\sigma}_t^{\phi} (y) dt \right]  \\
	&= \psi_{\mathcal{N}}  \left[ \int_{-\infty}^{\infty}  \hat{\mathcal{E}} \circ \hat{\sigma}_t^{\phi} (y) dt \right]  \\
	&= \psi_{\mathcal{N}}  \left[ \int_{-\infty}^{\infty} \hat{\sigma}_t^{{\phi_{\mathcal{N}}}}  \circ \hat{\mathcal{E}} (y) dt \right]  \\
	& = \tau_{\mathcal{N}}  [h_{\psi_{\mathcal{N}}} \hat{\mathcal{E}}(y)] =  \tau  [ \iota (h_{\psi_{\mathcal{N}}} )(y)],
	\end{split}
	\end{equation*}
	where the last equality uses \eqref{eq:5.14}. This completes our proof, since $\tau$ is faithful.
\end{proof}

\begin{theorem}\label{thm:sufficient}
	Let $\mathcal{N}$ be a von Neumann subalgebra of $\mathcal{M}$ and $\psi, \phi$ be n.f. states on $\mathcal{M}.$ Suppose that $\sigma_t^\phi (\mathcal{N}) = \mathcal{N}$, and $\mathcal{E}:\mathcal{M}\to\mathcal{N}$ is the induced conditional expectation. If $\psi_{\mathcal{N}} \circ \mathcal{E} = \psi$, then we have
	\begin{equation*}
	\tilde{D}_p  (\psi_{\mathcal{N}} \| \phi_{\mathcal{N}} ) = \tilde{D}_p  (\psi \| \phi ),~~p\in(0,1)\cup(1,\infty).
	\end{equation*}
\end{theorem}

\begin{proof}
	Obviously we have $\phi_{\mathcal{N}} \circ \mathcal{E} = \phi \circ \mathcal{E} = \phi.$ Thus using Proposition \ref{prop:embed} and the fact that $\iota$ naturally embeds $L_p(\mathcal{N}, \phi_{\mathcal{N}})$ into $L_p(\mathcal{M}, \phi)$ for all $0< p < \infty$ \cite{HJX2010}, we have
	\begin{equation*}
	\begin{split}
	\left\| h_{\phi_{\mathcal{N}}}^{\frac{1-p}{2p}} h_{\psi_{\mathcal{N}}} h_{\phi_{\mathcal{N}}}^{\frac{1-p}{2p}} \right\|_{L_p(\mathcal{N}, \phi_{\mathcal{N}})} 
	& = \left\| \iota \left( h_{\phi_{\mathcal{N}}}^{\frac{1-p}{2p}} h_{\psi_{\mathcal{N}}} h_{\phi_{\mathcal{N}}}^{\frac{1-p}{2p}} \right) \right\|_{L_p(\mathcal{M}, \phi)}\\
	& = \left\| h_{\phi_{\mathcal{N}} \circ \mathcal{E}}^{\frac{1-p}{2p}} h_{\psi_{\mathcal{N}} \circ \mathcal{E}} h_{\phi_{\mathcal{N}} \circ \mathcal{E}}^{\frac{1-p}{2p}} \right\|_{L_p(\mathcal{M}, \phi)}\\
	& =  \left\| h_{\phi}^{\frac{1-p}{2p}} h_{\psi} h_{\phi}^{\frac{1-p}{2p}} \right\|_{L_p(\mathcal{M}, \phi)}.\qedhere
	\end{split}
	\end{equation*}
\end{proof}

%%%%%%%%%%%

\subsection*{Acknowledgements}
We would like to thank Quanhua Xu and Simeng Wang for helpful comments and discussions. JG and ZY are partially supported by NSFC No. 11771106, No. 11431011 and No. 11826012. HZ is partially supported by the French project ISITE-BFC (contract ANR-15-IDEX-03) and the NCN (National Centre of Science) grant 2014/14/E/ST1/00525.

%%%%%%%%%%% To ease editing, use normal size for the references:

\end{document}